\documentclass[11pt, reqno, a4paper]{amsart} 

\usepackage{myArticleStyle}
\usepackage{myTheoremStyle}
\usepackage{myCommands}
\usepackage{myBibliographyStyle}

\DeclareMathOperator{\Lip}{Lip}
\newcommand{\HoldLipSpaceFD}[2]{\Lip_{\Omega}(#1;#2)}
\newcommand{\HoldLipSpace}[3]{\Lip_{#1}(#2;#3)}
\newcommand{\HoldLipSpaceMultFD}[2]{\Lip_{\Omega}^{M}(#1;#2)}

\newcommand{\dualLat}[1]{#1^{\perp}}

\newcommand{\diffOpBack}{\overline{\triangle}}
\newcommand{\diffOpForw}{\triangle}

\title[Titchmarsh theorems on fundamental domains]{Titchmarsh theorems for H\"older-Lipschitz functions on fundamental domains of lattices in \(\bR^{d}\) with applications to boundedness of Fourier multipliers}

\author[A. Hendrickx]{Arne Hendrickx \orcidlink{0000-0002-4537-8627}}

\address{
 Arne Hendrickx
  \endgraf
  Department of Mathematics: Analysis, Logic and Discrete Mathematics
  \endgraf
  Ghent University, Belgium
  \endgraf
  {\it E-mail address:} {\rm arnhendr.Hendrickx@UGent.be}
  }



\thanks{This research has been supported by Fonds Wetenschappelijk Onderzoek - Vlaanderen (FWO) with the PhD Fellowship Fundamental Research grant 1187323N}

\begin{document}

\begin{abstract}
    We extend the classical Titchmarsh theorems to the Fourier transform of two types of H\"older-Lipschitz functions -- additive and multiplicative -- defined on fundamental domains of lattices in $\mathbb{R}^d$. Our approach is based on generalizations of Duren's lemma, which we first illustrate in the classical Euclidean setting. As an application of the second Titchmarsh theorem, we obtain boundedness results for Fourier multipliers between H\"older-Lipschitz spaces, from which we deduce Lipschitz-Sobolev regularity for Bessel potential operators on fundamental domains of lattices in the additive case. These results provide a natural generalization of classical one-dimensional theorems on the real line and on the torus to higher dimensions.
\end{abstract}

\maketitle

\tableofcontents

\section{Introduction}

In 1948 Titchmarsh proved two theorems about H\"older-Lipschitz spaces on \(\bR\) in \cite{Titchmarsh}, which now bear his name. The first theorem is an extension of the Hausdorff-Young inequality $\lVert \FT{f}\rVert_{L^{p'}} \leq \normIn{f}{L^p}$ for \(1 \leq p \leq 2\), where $\frac{1}{p} + \frac{1}{p'} = 1$. As a consequence of this inequality, we see that $\FT{f} \in L^{p'}$. Titchmarsh extended the range of exponents \(\beta\) for which \(\FT{f} \in L^{\beta}\) under some Lipschitz-type condition on \(f \in L^{p}\).

\begin{theorem}[{\cite[Theorem 84]{Titchmarsh}}]
	Suppose that \(f \in L^{p}(\bR)\), for some \(1 < p \leq 2\), satisfies the Lipschitz-type condition
	\[\int_{-\infty}^{\infty} \abs{f(x+h)-f(x-h)}^{p} \, \mathrm{d}x = O\left(h^{\alpha p}\right) \quad \text{as} \quad h \to 0.\]
	Then its Fourier transform \(\FT{f}\) belongs to \(L^{\beta}(\bR)\) for
	\[\frac{p}{p+\alpha p -1} < \beta \leq \frac{p}{p-1}.\]
\end{theorem}

The second theorem of Titchmarsh is a characterisation of \(L^{2}\)-functions satisfying a certain Lipschitz-type condition in terms of an asymptotic estimate for an integral over its Fourier transform.

\begin{theorem}[{\cite[Theorem 85]{Titchmarsh}}]
	Suppose \(f \in L^{2}(\bR)\) and \(0 < \alpha < 1\). Then the condition
	\[\int_{-\infty}^{\infty} \abs{f(x+h)-f(x-h)}^{2} \, \mathrm{d}x = O\left(\abs{h}^{2\alpha}\right) \quad \text{as} \quad h \to 0\]
	is equivalent with
	\[\left(\int_{-\infty}^{-X} + \int_{X}^{\infty}\right) |\FT{f}(\xi)|^{2} \, \mathrm{d}\xi = O\left(X^{-2\alpha}\right) \quad \text{as} \quad X \to \infty.\]
\end{theorem}

These theorems have been extended to several different settings. Younis considered Titchmarsh theorems on the (one- and two-dimensional) torus and on compact zero-dimensional groups in \cite{LipThesis} using Duren's lemma \cite{Duren}. He conducted further studies in \cite{LipDom} on the comparison of the Fourier transform of certain Lipschitz functions with the Hankel transforms of these functions and with their Fourier transforms on the Euclidean Cartan motion group \(M(n)\) for \(n \geq 2\), and in \cite{TitchMarshHyperbolic} on Titchmarsh-like theorems for complex-valued functions on the non-Euclidean hyperbolic place.

Titchmarsh theorems on compact homogeneous manifolds were studied in \cite{LipArticle} with a new approach for proving the first Titchmarsh theorem. An alternative Lipschitz condition based on spherical mean operators was considered in \cite{Bray}. Analogues of the classical Titchmarsh theorem have been established on the image under the discrete Fourier–Laplace transform of a set of functions satisfying a generalized Lipschitz condition in the space $L^2$ on the sphere \cite{TitchmarshSpherical}, on symmetric spaces of rank $1$ \cite{Platonov}, on Damek-Ricci spaces \cite{DamekRicci}, on harmonic $NA$ groups \cite{TitchmarshNAgroups}, in a Laguerre hypergroup \cite{LaguerreHypergroup}, and for the Fourier-Walsh transform on $[0, \infty)$ \cite{FourierWalsh}.

We are primarily concerned with generalizing the Titchmarsh theorems to the setting of fundamental domains of lattices. A fundamental domain of a lattice \(L = A\bZ^{d}\), where the generator matrix \(A\) is an invertible matrix, is a measurable set \(\Omega \subseteq \bR^{d}\) such that \(\Omega + L = \bR^{d}\) as a direct sum. Fuglede's theorem ensures that \(\{e^{2\pi i \kappa \cdot x} \mid \kappa \in \dualLat{L}\}\) is an orthonormal basis for \(L^{2}(\Omega)\), which we endow with the normalised measure \(\mathrm{d}x/\abs{\Omega}\), where the dual lattice is defined as \(\dualLat{L} := \left(\transp{A}\right)^{-1} \bZ^{d}\). In this way one can define the Fourier transform \(\cF_{\Omega}: L^{2}(\Omega) \to \ell^{2}(\dualLat{L})\) by
\[\cF_{\Omega}{f}(\kappa) := \cF_{\Omega} f(\kappa) := \frac{1}{\abs{\Omega}} \int_{\Omega} f(x) \, e^{-2 \pi i \kappa \cdot x} \, \d{x}.\]
Plancherel's formula
\begin{equation}\label{eq:Plancherel}
	\normIn{f}{L^{2}(\Omega)}^{2} = \lVert{\FT{f}}\rVert_{\ell^{2}(\dualLat{L})}^{2}
\end{equation}
for \(f \in L^{2}(\Omega)\) follows directly from this construction.
Moreover, it has been shown using interpolation techniques that the Hausdorff-Young inequality also holds in this setting.
\begin{theorem}[Hausdorff-Young inequality]
	Let \(1 \leq p \leq 2\) and \(\frac{1}{p} + \frac{1}{p'} = 1\). If \(f \in L^{p}(\Omega)\), then \(\FT{f} \in \ell^{p'}(\dualLat{L})\) and
	\begin{equation}\label{ineq:HY}
		\lVert{\FT{f}}\rVert_{\ell^{p'}(\dualLat{L})} \leq \normIn{f}{L^{p}(\Omega)}.
	\end{equation}
\end{theorem}
A more detailed outline of the necessary concepts can be found in \cite{LpLqBoundedness}.

In this setting we introduce two different types of H\"older-Lipschitz spaces, which we will call \emph{additive} and \emph{multiplicative}. The additive H\"older-Lipschitz functions satisfy a Lipschitz-type condition of the form $\normIn{f(\cdot + h) - f(\cdot)}{L^p(\Omega)} = O\left(\abs{h}^{\alpha}\right)$, and they are called additive because the error term is equivalent to $O(h_{1}^{\alpha} + \dots + h_{d}^{\alpha})$. On the other hand, multiplicative H\"older-Lipschitz spaces satisfy a Lipschitz-type condition of the form $\normIn{\diffOpForw_{1}^{h_1} \dots \diffOpForw_{d}^{h_d} f}{L^p(\Omega)} = O\left(h_{1}^{\alpha_1} \dots h_{d}^{\alpha_d}\right)$, where $\diffOpForw_{j}^{h_j}$ is the difference operator with respect to the $j$-th coordinate with step $h_j$, and are called as such because the error term consists of a multiplication of coordinate-wise error terms.

Our proof strategies for the Titchmarsh theorems in this new setting is based on applications of Duren-type lemmas, which is firstly illustrated in \cref{s:TitchRd} with a new approach for the known cases in $\bR^d$ via continuous Duren-type lemmas. Adapting a multidimensional version of Duren's lemma to the lattice setting, we obtain for the additive Hölder-Lipschitz spaces the following form for the first Titchmarsh theorem.
\begin{itheorem}
	If \(f \in \HoldLipSpaceFD{\alpha}{p}\) for some \(0 < \alpha \leq 1\) and \(1 < p \leq 2\), then $\widehat{f} \in \ell^{\gamma}(\dualLat{L})$ for
	\[\frac{p}{p + \frac{\alpha}{d} p - 1} < \gamma \leq \infty.\]
\end{itheorem}
We discuss the second Titchmarsh theorem via two approaches originating from \cite{LipThesis} and \cite{LipArticle}.
\begin{itheorem}
	Let \(0 < \alpha < 1\) and \(f \in L^{2}(\Omega)\). Then \(f \in \HoldLipSpaceFD{\alpha}{2}\) if and only if
	\begin{equation}
		\sum_{\abs{\kappa} > N} |\FT{f}(\kappa)|^{2} = O\left(N^{-2\alpha}\right).
	\end{equation}
\end{itheorem}
An application of the second Titchmarsh theorem gives the following boundedness result of Fourier multipliers on additive H\"older-Lipschitz spaces.
\begin{itheorem}
	Let \(\gamma \in \bR\) with \(0 \leq \gamma < 1\). Suppose that for some $C > 0$ the function \(\sigma : \dualLat{L} \to \bC\) satisfies the growth estimate
	\[\abs{\sigma(\kappa)} \leq C \langle \kappa \rangle^{-\gamma} \quad \text{with} \quad \langle \kappa \rangle := (1+\abs{\kappa}^{2})^{\frac{1}{2}}.\]
	Let \(T\) be the Fourier multiplier with symbol \(\sigma\), i.e., \(\FT{Tf}(\kappa) = \sigma(\kappa) \FT{f}(\kappa)\) for all \(\kappa \in \dualLat{L}\). Then
	\[T : \HoldLipSpaceFD{\alpha}{2} \to \HoldLipSpaceFD{\alpha + \gamma}{2}\]
	is bounded for every \(\alpha \in \bR\) with \(0 < \alpha < 1 - \gamma\).
\end{itheorem}
Similar results are discussed for the multiplicative H\"older-Lipschitz spaces.

\subsection{Organization of the paper}

In \cref{s:TitchRd} we discuss Titchmarsh's proof of the first Titchmarsh theorem in \cite{Titchmarsh}. We zoom in on a particular step in the proof that has been left rather vague, and fill this gap with a continuous versions of Duren's lemma. This lemma has been extended to higher dimensions in \cref{lemma:DurenContAdd,lemma:DurenContMult}, which are fit for the proofs of the first Titchmarsh theorem for the additive and multiplicative H\"older-Lipschitz spaces in $\bR^d$. The approach for proving the first Titchmarsh theorem with the help of continuous versions of Duren's lemma will be exemplary for the proof strategies later on in the setting of fundamental domains of lattices.

The Titchmarsh theorems for additive and multiplicative H\"older-Lipschitz spaces are treated in \cref{s:TitchThFundDom} via lattice versions of the multidimensional Duren lemma. We discuss both the proof strategy from \cite{LipThesis} and \cite{LipArticle} for the first Titchmarsh theorem in the additive case. Sharpness for the first Titchmarsh theorem is discussed in both the additive and multiplicative case.

As an application of the second Titchmarsh theorem, boundedness results for Fourier multipliers on H\"older-Lipschitz space with respect to a norm related to the asymptotic estimate in the second Titchmarsh theorem are derived in \cref{s:FourierMult}. In the additive case we deduce Lipschitz-Sobolev regularity for Bessel potential operators on fundamental domains of lattices from this boundedness result, and we discuss the relation of our newly introduced norm with another one that appears in the literature via a refined formulation of Duren's lemma.

\subsection{Notation and conventions}

We follow the convention that \(0 \in \bN\).

Throughout this paper \(L\) stands for a lattice in \(\bR^{d}\), and \(\Omega\) denotes a fundamental domain of \(L\).

We denote the conjugate exponent of \(1 \leq p \leq \infty\) by \(p'\), i.e., \(\frac{1}{p} + \frac{1}{p'} = 1\).

Let \(f,g : \bR^n \to \bR\) be two functions for some $n \in \bN$. We use the big $O$ notation $f(x) = O(g(x))$ as $x \to a \in \bR \cup \{\infty\}$ to denote that there exists a constant $C < 0$ and a neighborhood $U$ of $a$ such that $f(x) \leq C g(x)$ for all $x \in U$. We sometimes denote this by \(f(x) \lesssim g(x)\) if the point $a$ is known from the context.

We define the floor function $\floor{x}$ as the greatest integer less than or equal to $x \in \bR$.

For vectors \(x,y \in \bR^{d}\) we write \(x \cdot y = \sum_{j=1}^{d} x_{j} y_{j}\) for the Euclidean inner product and \(\abs{x} = \sqrt{x \cdot x}\) for the Euclidean norm.
We denote by \(\abs{\cdot}_{p}\) the \(p\)-norm on \(\bR^{d}\) for \(1 \leq p \leq \infty\), i.e.,
\[\abs{x}_{p} := \begin{dcases}
    \left(\sum_{j=1}^{d} \abs{x_{j}}^{p}\right)^{1/p} &\quad \text{if} \quad 1 \leq p < \infty \\
    \max_{1 \leq j \leq d} \abs{x_{j}} &\quad \text{if} \quad p = \infty
\end{dcases}.\]
We simply write \(\abs{\cdot}\) for \(\abs{\cdot}_{2}\).

\section{The first Titchmarsh theorem in \texorpdfstring{$\bR$}{bR} via continuous Duren-type lemmas}\label{s:TitchRd}

The main approach for proving Titchmarsh theorems in this paper is centered around Duren's lemma, which was proven in a discrete setting \cite[101]{Duren}. We will firstly illustrate this approach in a continuous setting by giving a proof of the first Titchmarsh theorem in the case of H\"older-Lipschitz function on $\bR$.

In his proof of the original first Titchmarsh theorem \cite[Theorem 84]{Titchmarsh}, Titchmarsh obtained that
\[\varphi(\xi) := \int_{1}^{\xi} \abs{x}^{\beta} \, |\FT{f}(x)|^{\beta} \, \d{x} = O\left( \xi^{1 - \alpha \beta + \frac{\beta}{p}} \right)\]
under the assumptions that $1 < p \leq 2$, $0 < \alpha \leq 1$ and $1 \leq \beta < p'$, where $\FT{f}$ is the Fourier transform of $f \in L^p(\bR)$. Then he calculates that
\[\int_{1}^{\xi} |\FT{f}(x)|^{\beta} \, \d{x} = O\left( \xi^{1 - \beta - \alpha \beta + \frac{\beta}{p}} \right),\]
from which he concluded that $\int_{1}^{\infty} |\FT{f}(x)| \, \d{x}$ converges if $1 - \beta - \alpha \beta + \frac{\beta}{p} \leq 0$. This last deduction is somewhat erroneous since $\int_{1}^{\xi} |\FT{f}(x)|^{\beta} \, \d{x} = O\left( \xi^{1 - \beta - \alpha \beta + \frac{\beta}{p}} \right)$ with $1 - \beta - \alpha \beta + \frac{\beta}{p} < 0$ would imply that $\int_{1}^{\infty} |\FT{f}(x)| \, \d{x} = 0$ so that $\FT{f} = 0$ on $[1,\infty]$. The error is that the $O$-estimate is assumed to be valid globally instead of only for sufficiently large $\xi$. Instead, one should look for sufficiently large $\xi$ at
\[\int_{\xi}^{\infty} |\FT{f}(x)|^{\beta} \d{x} = \int_{\xi}^{\infty} x^{-\beta} \, \varphi'(x) \, \d{x} = \lim_{x \to \infty} x^{-\beta} \varphi(x) - \xi^{-\beta} \varphi(\xi) + \beta \int_{\xi}^{\infty} x^{-\beta-1} \varphi(x) \, \d{x}.
\]
If $1 - \alpha \beta + \frac{\beta}{p} < \beta$, then $\lim_{x \to \infty} x^{-\beta} \varphi(x) = 0$ so that
\[\int_{\xi}^{\infty} |\FT{f}(x)|^{\beta} \, \d{x} \leq \beta \int_{\xi}^{\infty} O\left(x^{-\beta - \alpha \beta + \frac{\beta}{p}}\right) \, \d{x} = O\left(\xi^{1 - \beta - \alpha \beta + \frac{\beta}{p}}\right).\]
We can do the same calculation for $\int_{-\infty}^{-\xi} |\FT{f}(x)| \, \d{x}$. From this, we can indeed conclude that $\int_{-\infty}^{\infty} |\FT{f}(x)|^{\beta} \, \d{x}$ converges if $1 - \beta - \alpha \beta + \frac{\beta}{p} < 0$.

The essence of this corrected approach, where we go from an estimate for the integral $\int_{1}^{\xi} \abs{x}^{\beta} |\FT{f}(x)|^{\beta} \, \d{x}$ over a bounded domain to an estimate for the integral $\int_{\xi}^{\infty} |\FT{f}(x)|^{\beta} \, \d{x}$ over an unbounded domain, can be distilled into the following continuous version of Duren's lemma.

\begin{lemma}[Continuous Duren-type lemma]\label{lemma:DurenCont}
	Consider a non-negative function $f \in L^1[0,\infty)$, and let $0 < a < b$. The estimate
	\begin{equation}\label{eq:DurenCont1}
		\int_{0}^{X} x^b \, f(x) \, \d{x} = O(X^a)
	\end{equation}
	is equivalent with
	\begin{equation}\label{eq:DurenCont2}
		\int_{X}^{\infty} f(x) \, \d{x} = O(X^{a-b}).
	\end{equation}
\end{lemma}

\begin{proof}
	Suppose that \eqref{eq:DurenCont1} holds. Set $\varphi(X) := \int_{0}^{X} x^b \, f(x) \, \d{x}$. We find for sufficiently large $X < Y$ that
	\begin{align*}
		\int_{X}^{Y} f(x) \, \d{x} = \int_{X}^{Y} x^{-b} \, \varphi'(x) \, \d{x}
		&= Y^{-b} \varphi(Y) - X^{-b} \varphi(X) + b \int_{X}^{Y} x^{-b-1} \, \varphi(x) \, \d{x} \\
		&\lesssim Y^{a-b} + b \int_{X}^{Y} x^{a-b-1} \, \d{x}.
	\end{align*}
	Hence, in the limit $Y \to \infty$ we get
	\[\int_{X}^{\infty} f(x) \, \d{x} \lesssim b \int_{X}^{\infty} x^{a-b-1} \, \d{x} = \frac{b}{b-a} X^{a-b}.\]
	
	Conversely, assume that \eqref{eq:DurenCont2} holds. Set $\psi(X) := \int_{X}^{\infty} f(x) \, \d{x}$. Since $\psi(X) = O(X^{a-b})$, there is some $X_0 > 0$ and $C > 0$ such that for all $X \geq X_0$ we have $\psi(X) \leq C X^{a-b}$. We then find that
	\begin{align*}
		\int_{0}^{X} x^b f(x) \, \d{x}
		= - \int_{0}^{X} x^b \, \psi'(x) \, \d{x}
		&= -X^b \psi(X) + b \int_{0}^{X} x^{b-1} \psi(x) \, \d{x} \\
		&\leq \int_{0}^{X_0} x^{b-1} \psi(x) \, \d{x} + C \int_{X_0}^{X} x^{a-1} \, \d{x} \\
		&= O(X^a),
	\end{align*}
	completing the proof.
\end{proof}

As preparation for the proof of the first Titchmarsh theorem, let us recall an inequality that we will rely on.

\begin{lemma}[{Jordan's inequality, \cite[Inequality (1.1)]{JordanInequality}}]
	For \(\abs{t} \leq \frac{1}{2}\) we have the following inequality:
	\begin{equation}\label{ineq:JordanIneq}
		2\abs{t} \leq \abs{\sin \pi t} \leq \pi \abs{t}.
	\end{equation}
\end{lemma}

For convenience of the reader and to set the stage for our proof strategies later, we will now provide a cleaned up version of the proof of the first Titchmarsh theorem with the help of the continuous version of Duren's lemma.

\begin{theorem}[{\cite[Theorem 84]{Titchmarsh}}]
	Suppose that \(f \in L^{p}(\bR)\), for some \(1 < p \leq 2\), satisfies the Lipschitz-type condition
	\[\int_{-\infty}^{\infty} \abs{f(x+h)-f(x-h)}^{p} \, \mathrm{d}x = O\left(h^{\alpha p}\right) \quad \text{as} \quad h \to 0^{+}\]
	for some $0 < \alpha \leq 1$.
	Then its Fourier transform \(\FT{f}\) belongs to \(L^{\beta}(\bR)\) for
	\[\frac{p}{p+\alpha p -1} < \beta \leq p' = \frac{p}{p-1}.\]
\end{theorem}

\begin{proof}
	A straightforward calculation gives for $h > 0$ that
	\[\FT{f(\cdot + h) - f(\cdot - h)}(\xi) = (e^{2 \pi i \xi h} - e^{-2 \pi i \xi h}) \FT{f}(\xi) = 2 i \sin(2 \pi \xi h) \FT{f}(\xi)\]
	so that by Jordan's inequality \eqref{ineq:JordanIneq} and the Hausdorff-Young inequality we get
	\begin{align*}
		2^{3p'} \int_{-\frac{1}{4h}}^{\frac{1}{4h}} \abs{\xi}^{p'} \, h^{p'} \, |\FT{f}(\xi)|^{p'} \, \d{\xi}
		\leq 2^{p'} \int_{-\infty}^{\infty} \abs{\sin(2 \pi \xi h)}^{p'} \, |\FT{f}(\xi)|^{p'} \, \d{\xi}
		&= \normIn{f(\cdot+h)-f(\cdot-h)}{L^{p'}(\bR)}^{p'} \\
		&\leq \normIn{f(\cdot + h) - f(\cdot - h)}{L^p(\bR)}^{p'} \\
		&= O(h^{\alpha p'}).
	\end{align*}
	Hence, setting $X := \frac{1}{4h}$, we found that
	\[\int_{-X}^{X} \abs{\xi}^{p'} \, |\FT{f}(\xi)|^{p'} \, \d{\xi} = O(X^{(1-\alpha)p'}).\]
	Using H\"older's inequality with exponents $\frac{p'}{\beta}$ and $1/(1 - \frac{\beta}{p'})$ where $1 \leq \beta < p'$, we obtain that
	\[\int_{-X}^{X} \abs{\xi}^{\beta} |\FT{f}(\xi)|^{\beta} \, \d{\xi} \leq \left( \int_{-X}^{X} \abs{\xi}^{p'} |\FT{f}(\xi)|^{p'} \, \d{\xi} \right)^{\frac{\beta}{p'}} \left( \int_{-X}^{X} \d{\xi} \right)^{1 - \frac{\beta}{p'}} = O\left(X^{(1-\alpha)\beta + 1 - \frac{\beta}{p'}}\right).\]
	By \cref{lemma:DurenCont}, this is equivalent with
	\[\left( \int_{-\infty}^{-X} + \int_{X}^{\infty} \right) |\FT{f}(\xi)|^{\beta} \, \d{\xi} = O\left( X^{1 - \alpha \beta - \frac{\beta}{p'}} \right)\]
	under the condition that $0 < (1 - \alpha) \beta + 1 - \frac{\beta}{p'} < \beta$. The left inequality is satisfied as $0 < \alpha \leq 1$ and $1 \leq \beta < p'$, and the right one can be rewritten as
	\[\frac{p}{p + \alpha p - 1} = \frac{1}{\alpha + \frac{1}{p'}} < \beta.\]
	Note that under this condition we have exactly convergence of the integral $\int_{-\infty}^{\infty} |\FT{f}(\xi)|^{\beta} \, \d{\xi}$. The case $\beta = p'$ follows from the Hausdorff-Young inequality, completing the proof.
\end{proof}

A natural extension of the original first Titchmarsh theorem for H\"older-Lipschitz functions on $\bR$ consists of considering functions in several variables. This has been investigated, for example, in \cite[Theorem 2.15]{LipThesis}. However, the proof there contains the same inaccuracy as the proof of the original Titchmarsh theorem. We will give a clear proof again via a continuous version of Duren's lemma, which is adapted to several variables this time.

\begin{lemma}[Multidimensional continuous Duren-type lemma, additive case]\label{lemma:DurenContAdd}
	Let $d \in \bN \setminus \{0\}$. Consider a non-negative function $f \in L^1(\bR^d)$, and let $0 < a < b$. The estimate
	\[\int_{\abs{x} \leq X} \abs{x}^b \, f(x) \, \d{x} = O(X^{a})\]
	is equivalent with
	\[\int_{\abs{x} > X} f(x) \, \d{x} = O(X^{a-b}).\]
\end{lemma}

\begin{proof}
	Via spherical coordinates we have that
	\begin{equation}\label{eq:DurenContAdd1}
		\int_{0}^{X} r^{b} \left( \int_{S_r} r^{d-1} \, f(r \theta) \, d{\theta} \right) \, \d{r} = \int_{\abs{x} \leq X} \abs{x}^b \, f(x) \, \d{x} = O(X^a),
	\end{equation}
	where $S_r := \{x \in \bR^n : \abs{x} = r\}$ for $r > 0$ is the sphere centered at the origin of radius $r$. Via \cref{lemma:DurenCont} we then find that \eqref{eq:DurenContAdd1} is equivalent with
	\[\int_{\abs{x} > X} f(x) \, \d{x} = \int_{X}^{\infty} \left( \int_{S_r} r^{d-1} \, f(r \theta) \, d{\theta} \right) \, \d{r} = O(X^{a-b}),\]
	which is what we had to prove.
\end{proof}

We can now proceed to a proof of the first Titchmarsh theorem for H\"older-Lipschitz functions in $\bR^d$ via the former multidimensional continuous version of Duren's lemma. 

\begin{theorem}[{\cite[Theorem 2.15]{LipThesis}}]\label{th:firstTitchRdAdd}
	Let $f \in L^p(\bR^d)$ for some $1 < p \leq 2$, and suppose that
	\[\normIn{f(x_1 + h_1, \dots, x_d + h_d) - f(x_1, \dots, x_d)}{L^p(\bR^d)} = O\left( h_1^{\alpha_1} + \dots + h_d^{\alpha_d} \right) \quad \text{as} \quad h \to 0^{+}\]
	for some $0 < \alpha_1, \dots, \alpha_d \leq 1$, where $h \to 0^{+}$ means that $h_j \to 0^{+}$ for all $j$. Then $\FT{f} \in L^{\beta}(\bR^d)$ for
	\[\frac{p}{p + \frac{\alpha}{d} p - 1} < \beta \leq \frac{p}{p-1},\]
	where $\alpha := \min_{1 \leq j \leq d} \alpha_j$.
\end{theorem}

\begin{proof}
	We directly calculate that
	\[\FT{f(\cdot + h) - f(\cdot)}(\xi) = \left( e^{2 \pi i \xi \cdot h} - 1 \right) \FT{f}(\xi) = 2 i \sin(\pi \xi \cdot h) e^{i \pi \xi \cdot h} \FT{f}(\xi)\]
	so that using the Hausdorff-Young inequality gives
	\[2^{p'} \int_{\abs{\xi} \leq \frac{1}{2\abs{h}}} \abs{\sin(\pi \xi \cdot h)}^{p'} |\FT{f}(\xi)|^{p'} \d\xi \leq \normIn{f(\cdot + h) - f(\cdot)}{L^p(\bR^d)}^{p'} = O(h_1^{\alpha_1 p'} + \dots + h_d^{\alpha_d p'}).\]
	Let $X > 0$ and define $h_i := \frac{1}{2X} e_i$ for $1 \leq i \leq d$. Then Jordan's inequality \eqref{ineq:JordanIneq} gives that
	\[\sum_{i=1}^{d} \abs{\sin(\pi \xi \cdot h_i)}^{p'} \geq \left(\frac{2}{2X}\right)^{p'} \sum_{i=1}^{d} \abs{\xi_i}^{p'} \geq \left(\frac{1}{X}\right)^{p'} d^{\frac{1}{2} - \frac{1}{p'}} \abs{\xi}^{p'}\]
	so that
	\begin{align*}
		\int_{\abs{\xi} \leq X} \abs{\xi}^{p'} \, |\FT{f}(\xi)|^{p'} \, \d\xi
		&\leq X^{p'} d^{\frac{1}{p'} - \frac{1}{2}} \sum_{i=1}^{d} \int_{\abs{\xi} \leq \frac{1}{2\abs{h_i}}} \abs{\sin(\pi \xi \cdot h_i)}^{p'} \, |\FT{f}(\xi)|^{p'} \, \d\xi \\
		&= O\left( X^{(1-\alpha_1)p'} + \dots + X^{(1-\alpha_d)p'} \right) \\
		&= O\left( X^{(1-\alpha)p'} \right).
	\end{align*}
	Applying Hölder's inequality with exponents $\frac{p'}{\beta}$ and $1 / (1 - \frac{\beta}{p'})$ where $1 \leq \beta < p'$ gives that
	\[\int_{\abs{\xi} \leq X} \abs{\xi}^{\beta} \, |\FT{f}(\xi)|^{\beta} \, \d\xi
	\leq \left( \int_{\abs{\xi} \leq X} \abs{\xi}^{p'} \, |\FT{f}(\xi)|^{p'} \, \d\xi \right) \left( \int_{\abs{\xi} \leq X} \d\xi \right)^{1 - \frac{\beta}{p'}} = O\left( X^{(1 - \alpha) \beta + d \left( 1 - \frac{\beta}{p'} \right)} \right).\]
	An application of \cref{lemma:DurenContAdd} gives that
	\[\int_{\abs{\xi} > X} |\FT{f}(\xi)|^{\beta} \, \d\xi = O\left( X^{d - \beta \left( \alpha + \frac{d}{p'} \right)}\right)\]
	under the condition that $0 < (1 - \alpha) \beta + d \left( 1 - \frac{\beta}{p'} \right) < \beta$. Note that the left inequality is satisfied because of the conditions $0 < \alpha \leq 1$ and $1 \leq \beta < p'$, while the right equality can be rewritten as
	\[\frac{d}{\alpha + \frac{d}{p'}} = \frac{p}{p + \frac{\alpha}{d} p - 1} < \beta.\]
	The case $\beta = p'$ follows from the Hausdorff-Young inequality.
\end{proof}

While the former one-dimensional and multidimensional versions of the first Titchmarsh theorems appeared in the literature, the following version of the first Titchmarsh theorem in several dimensions seems to be rarely studied in the case of $\bR^d$. It appears, however, in the toroidal case as studied in \cite{LipThesis} in the context of the second Titchmarsh theorem. For completeness, and since it corresponds to another multidimensional continuous version of Duren's lemma, we will treat it here.

\begin{lemma}[Multidimensional continuous Duren-type lemma, multiplicative case]\label{lemma:DurenContMult}
	Let $d \in \bN \setminus \{0\}$. Consider a non-negative function $f \in L^1(\bR^d)$, and let $0 < a_n < b_n$ for every $1 \leq n \leq d$. Then all estimates of the following form are equivalent:
	\begin{multline*}
		\int_{I_{\varepsilon_{1}}(X_1)} \dots \int_{I_{\varepsilon_{d}}(X_d)} \abs{x_1}^{\varepsilon_{1} b_1} \dots \abs{x_d}^{\varepsilon_{d} b_d} f(x_1, \dots, x_d) \, \d{x_1} \dots \d{x_d} \\
		= O\left( X_{1}^{a_1 - (1 - \varepsilon_{1}) b_1} \dots X_{d}^{a_d - (1 - \varepsilon_{d}) b_d} \right),
	\end{multline*}
	where $\varepsilon_{1}, \dots, \varepsilon_{d} \in \{0,1\}$, $I_{1}(X) := [-X,X]$ and $I_{0}(X) := \bR \setminus [-X,X]$ for $X > 0$.
\end{lemma}

\begin{proof}
	Note that it suffices to prove this result integral by integral, and without loss of generality we can prove it for the first integral, i.e.,
	\begin{multline}\label{eq:DurenContMult1}
		\int_{-X_1}^{X_1} \int_{I_{\varepsilon_2}(X_2)} \dots \int_{I_{\varepsilon_{d}}(X_d)} \abs{x_1}^{b_1} \abs{x_2}^{\varepsilon_2 b_2} \dots \abs{x_d}^{\varepsilon_{d} b_d} f(x_1, \dots, x_d) \, \d{x_1} \dots \d{x_d} \\
		= O\left( X_{1}^{a_1} X_{2}^{a_2 - (1 - \varepsilon_2) b_2} \dots X_{d}^{a_d - (1 - \varepsilon_{d}) b_d} \right)
	\end{multline}
	is equivalent to
	\begin{multline}\label{eq:DurenContMult2}
		\int_{\abs{x_1} > X_1} \int_{I_{\varepsilon_2}(X_2)} \dots \int_{I_{\varepsilon_{d}}(X_d)} \abs{x_2}^{\varepsilon_2 b_2} \dots \abs{x_d}^{\varepsilon_{d} b_d} f(x_1, \dots, x_d) \, \d{x_1} \dots \d{x_d} \\
		= O\left( X_{1}^{a_1 - b_1} X_{2}^{a_2 - (1 - \varepsilon_2) b_2} \dots X_{d}^{a_d - (1 - \varepsilon_{d}) b_d} \right).
	\end{multline}
	For ease of notation, set
	\[F_{X_2, \dots X_d}(x_1) := \int_{I_{\varepsilon_2}(X_2)} \dots \int_{I_{\varepsilon_{d}}(X_d)} \abs{x_2}^{\varepsilon_2 b_2} \dots \abs{x_d}^{\varepsilon_{d} b_d} f(x_1, \dots, x_d) \, \d{x_2} \dots \d{x_d}.\]
	One can verify in a straightforward way that $F_{X_2, \dots, X_d} \in L^1(\bR)$ for every $X_2, \dots, X_d > 0$ since $f \in L^1(\bR^d)$.
	
	Now, fix sufficiently large $X_2, \dots, X_d > 0$. We need to prove that
	\[\int_{-X_1}^{X_1} x_{1}^{b_1} \frac{F_{X_2, \dots, X_d}(x_1)}{X_{2}^{a_2 - (1 - \varepsilon_2) b_2} \dots X_{d}^{a_d - (1 - \varepsilon_{d}) b_d}} \, \d{x_1} = O\left( X_1^{a_1} \right)\]
	is equivalent to
	\[\int_{\abs{x_1} > X_1} \frac{F_{X_2, \dots, X_d}(x_1)}{X_{2}^{a_2 - (1 - \varepsilon_2) b_2} \dots X_{d}^{a_d - (1 - \varepsilon_{d}) b_d}} \, \d{x_1} = O\left( X_1^{a_1 - b_1} \right),\]
	but this follows immediately from \cref{lemma:DurenCont} because
	\[\frac{F_{X_2, \dots, X_d}}{X_2^{a_2 - (1 - \varepsilon_2) b_2} \dots X_d^{a_d - (1 - \varepsilon_{d}) b_d}} \in L^1(\bR)\]
	as $X_2, \dots, X_d > 0$ were fixed.
\end{proof}

This version of Duren's lemma enables us to treat the following ``multiplicative'' version of the first Titchmarsh theorem in $\bR^d$. We call it multiplicative because the Lipschitz condition corresponding to all individual coordinates gives rise to a product of sine factors in the Fourier space, and the $O$-estimate consists of a product of ``errors'' from each coordinate.

\begin{theorem}\label{th:firstTitchMultRd}
	Let $f \in L^p(\bR^d)$ for some $1 < p \leq 2$. Define the difference operator $\diffOpForw_{j}^{h_{j}}$ for each $1 \leq j \leq d$ with $h_{j} > 0$ by
	\[\diffOpForw_{j}^{h_{j}} f(x) := f(x_1, \dots, x_j + h_j, \dots, x_d) - f(x_1, \dots, x_j, \dots x_d).\]
	Suppose that
	\[\normIn{\diffOpForw_{1}^{h_{1}} \dots \diffOpForw_{d}^{h_{d}} f}{L^p(\bR^d)} = O\left( h_{1}^{\alpha_1} \dots h_{d}^{\alpha_d}\right) \quad \text{as} \quad h \to 0^{+}\]
	for some $0 < \alpha_1, \dots, \alpha_d \leq 1$. Then $\FT{f} \in L^{\beta}(\bR^{d})$ for
	\[\frac{p}{p + \alpha p - 1} < \beta \leq \frac{p}{p-1},\]
	where $\alpha := \min_{1 \leq j \leq d} \alpha_j$.
\end{theorem}

\begin{proof}
	After calculating for $X_1, \dots, X_d > 0$ that
	\begin{align*}
		\FT{\diffOpForw_{1}^{\frac{1}{2X_{1}}} \dots \diffOpForw_{d}^{\frac{1}{2X_{d}}} f}(\xi)
		&= \left(e^{2 \pi i \frac{\xi_{1}}{2X_{1}}} - 1\right) \dots \left(e^{2 \pi i \frac{\xi_{d}}{2X_{d}}} - 1\right) \FT{f}(\xi) \\
		&= 2 i \sin\left(\frac{\pi \xi_{1}}{2X_{1}}\right) \, e^{i \pi \frac{\xi_{1}}{2X_{1}}} \dots 2 i \sin\left(\frac{\pi \xi_{d}}{2X_{d}}\right) \, e^{i \pi \frac{\xi_{d}}{2X_{d}}} \FT{f}(\xi),
	\end{align*}
	we can apply Jordan's inequality \eqref{ineq:JordanIneq} and the Hausdorff-Young inequality to get
	\begin{align*}
		2^{d p'} \int_{|\xi_{1}| \leq X_{1}} \dots \int_{|\xi_{d}| \leq X_{d}} \abs{\frac{\xi_{1}}{X_{1}}}^{p'} \dots \abs{\frac{\xi_{d}}{X_{d}}}^{p'} |\FT{f}(\xi)|^{p'} \, \d\xi
		&\leq \normIn{\diffOpForw_{1}^{\frac{1}{2X_{1}}} \dots \diffOpForw_{d}^{\frac{1}{2X_{d}}} f}{L^{p}(\bR^d)}^{p'} \\
		&= O\left(X_{1}^{-\alpha p'} \dots X_{d}^{-\alpha p'}\right).
	\end{align*}
	Thus, we have found the estimate
	\begin{equation}
		\int_{|\xi_{1}| \leq X_{1}} \dots \int_{|\xi_{d}| \leq X_{d}} \big|\xi_{1}\big|^{p'} \dots \big|\xi_{d}\big|^{p'} |\FT{f}(\xi)|^{p'} \, \d\xi
		= O\left(X_{1}^{(1-\alpha)p'} \dots X_{d}^{(1-\alpha)p'}\right).
	\end{equation}
	Using H\"older's inequality with exponents $\frac{p'}{\beta}$ and $1/(1 - \frac{\beta}{p'})$ where $1 \leq \beta < p'$,
	we get
	\begin{multline*}
		\int_{\abs{\xi_{1}} \leq X_1} \dots \int_{\abs{\xi_{d}} \leq X_d} \big|\xi_{1}\big|^{\beta} \dots \big|\xi_{d}\big|^{\beta} |\FT{f}(\xi)|^{\beta} \, \d\xi \\
		\leq \left( \int_{|\xi_{1}| \leq X_{1}} \dots \int_{|\xi_{d}| \leq X_{d}} \big|\xi_{1}\big|^{p'} \dots \big|\xi_{d}\big|^{p'} |\FT{f}(\xi)|^{p'} \, \d\xi \right)^{\frac{\beta}{p'}} \left( \int_{|\xi_{1}| \leq X_{1}} \dots \int_{|\xi_{d}| \leq X_{d}} \d\xi \right)^{1 - \frac{\beta}{p'}} \\
		= O\left(X_1^{(1 - \alpha) \beta + 1 - \frac{\beta}{p'}} \dots X_d^{(1 - \alpha) \beta + 1 - \frac{\beta}{p'}}\right).
	\end{multline*}
	Applying Duren's lemma gives
	\[
	\int_{|\xi_{1}| \geq X_1} \dots \int_{|\xi_{d}| \geq X_{d}} |\FT{f}(\xi)|^{\beta} \, \d\xi = O\left( X_{1}^{1 - \alpha \beta - \frac{\beta}{p'}} \dots X_{d}^{1 - \alpha \beta - \frac{\beta}{p'}} \right)
	\]
	under the condition that $0 < (1 - \alpha) \beta + 1 - \frac{\beta}{p'} < \beta$. Note that the left inequality is satisfied, and the right one is equivalent to
	\[\frac{1}{\alpha + \frac{1}{p'}} = \frac{p}{p + \alpha p - 1} < \beta.\]
	This completes the proof.
\end{proof}

In the rest of this paper, similar proof strategies and tools will be used to treat discrete versions of Duren's lemma and Titchmarsh theorems for H\"older-Lipschitz functions on fundamental domains of lattices in $\bR^d$.

\section{H\"older-Lipschitz functions on fundamental domains of lattices}\label{s:TitchThFundDom}

In this section we will deduce the Titchmarsh theorems for H\"older-Lipschitz functions on fundamental domains of lattices in $\bR^d$. The first Titchmarsh theorem is a generalization of the Hausdorff-Young inequality \cite[Theorem 3.1]{LpLqBoundedness} in the sense that it determines a range of Lebesgue spaces to which the Fourier transform of H\"older-Lipschitz functions belong. The second Titchmarsh theorem characterizes H\"older-Lipschitz functions with exponent \(2\) by the asymptotic behavior of the series for the \(\ell^{2}\)-norm of their Fourier transforms. The latter theorem will also yield a boundedness result for Fourier multipliers on H\"older-Lipschitz spaces, which will be treated in \cref{s:FourierMult}.

The Duren-type lemmas for the proofs of the Titchmarsh theorems in this section are based on the following multidimensional version of Duren's lemma \cite[101]{Duren}.

\begin{lemma}[Duren]\label{lemma:Duren}
	Let \(d \in \bN \setminus \{0\}\). Consider a non-negative sequence \(c_{n_{1}, \dots, n_{d}} \geq 0\), and let \(0 < a_{n} < b_{n}\) for every \(1 \leq n \leq d\). Then all estimates of the following form are equivalent:
	\[\sum_{n_1 \in I_{\varepsilon_{1}}(N_1)} \dots \sum_{n_d \in I_{\varepsilon_{d}}(N_d)} n_{1}^{\varepsilon_{1} b_{1}} \dots n_{d}^{\varepsilon_{d} b_{d}} \, c_{n_{1}, \dots, n_{d}} = O\left(N_{1}^{a_{1}-(1-\varepsilon_{1})b_{1}} \dots N_{d}^{a_{d}-(1-\varepsilon_{d})b_{d}}\right),\]
	where \(\varepsilon_{1}, \dots, \varepsilon_{d} \in \{0,1\}\), \(I_{1}(N) := [1,N]\) and \(I_{0}(N) := \bR \setminus [0,N]\) for \(N \in \bN\).
	In particular, we have
	\[\sum_{n_{1} = 1}^{N_{1}} \dots \sum_{n_{d} = 1}^{N_{d}} n_{1}^{b_{1}} \dots n_{d}^{b_{d}} \, c_{n_{1}, \dots, n_{d}} = O\left(N_{1}^{a_{1}} \dots N_{d}^{a_{d}}\right)\]
	if and only if
	\[\sum_{n_{1} > N_{1}} \dots \sum_{n_{d} > N_{d}} c_{n_{1}, \dots, n_{d}} = O\left(N_{1}^{a_{1}-b_{1}} \dots N_{d}^{a_{d}-b_{d}}\right).\]
\end{lemma}

\begin{proof}
	The case \(d=1\) is treated in \cite[page 101]{Duren}, and a refined version of this one-dimensional case is treated by \cref{lemma:DurenRefined}. The case \(d = 2\) is worked out in \cite[Lemma 2.5]{LipThesis}, and we will follow an analogous strategy for higher dimensions. 
	
	It suffices to prove this lemma sum by sum, and without loss of generality we can prove it for the first sum, i.e.,
	\begin{multline*}
		\sum_{n_1 = 1}^{N_1} \sum_{n_2 \in I_{\varepsilon_2}(N_2)} \dots \sum_{n_d \in I_{\varepsilon_{d}}(N_d)} n_1^{b_1} n_2^{\varepsilon_2 b_2} \dots n_d^{\varepsilon_{d} b_d} c_{n_1, \dots, n_d} \\
		= O\left( N_1^{a_1} N_2^{a_2 - (1 - \varepsilon_2) b_2} \dots N_d^{a_d - (1 - \varepsilon_{d}) b_d} \right)
	\end{multline*}
	is equivalent to
	\begin{multline*}
		\sum_{n_1 > N_1} \sum_{n_2 \in I_{\varepsilon_2}(N_2)} \dots \sum_{n_d \in I_{\varepsilon_{d}}(N_d)} n_2^{\varepsilon_2 b_2} \dots n_d^{\varepsilon_{d} b_d} c_{n_1, \dots, n_d} \\
		= O\left( N_1^{a_1 - b_1} N_2^{a_2 - (1 - \varepsilon_2) b_2} \dots N_d^{a_d - (1 - \varepsilon_{d}) b_d} \right).
	\end{multline*}
	For ease of notation, let us set
	\[C_{n_1, N_2, \dots, N_d} := \sum_{n_2 \in I_{\varepsilon_2}(N_2)} \dots \sum_{n_d \in I_{\varepsilon_{d}}(N_d)} n_2^{\varepsilon_2 b_2} \dots n_d^{\varepsilon_{d} b_d} c_{n_1, \dots, n_d}.\]
	Fix sufficiently large $N_2, \dots, N_d \in \bN$. We need to prove the equivalence of the estimate
	\[\sum_{n_1 = 1}^{N_1} n_1^{b_1} \frac{C_{n_1, N_2, \dots, N_d}}{N_2^{a_2 - (1 - \varepsilon_2) b_2} \dots N_d^{a_d - (1 - \varepsilon_{d}) b_d}} = O\left( N_1^{a_1} \right)\]
	and
	\[\sum_{n_1 > N_1} \frac{C_{n_1, N_2, \dots, N_d}}{N_2^{a_2 - (1 - \varepsilon_{2}) b_2} \dots N_d^{a_d - (1 - \varepsilon_{d}) b_d}} = O\left( N_1^{a_1 - b_1} \right),\]
	which immediately follows from the one-dimensional case.
\end{proof}

Let us also point out a distinctive feature of working in fundamental domains of lattices. Since the dual group of the fundamental domain, namely the lattice itself, is discrete, we can extend the right bound in the range in the first Titchmarsh theorem to infinity as we have the following embedding between $\ell^p(\dualLat{L})$-spaces.

\begin{proposition}\label{prop:lpEmbedding}
	Let $1 \leq p < q \leq \infty$. Then the continuous embedding $\ell^p(\dualLat{L}) \hookrightarrow \ell^q(\dualLat{L})$ holds.
\end{proposition}

\begin{proof}
	Suppose $f \in \ell^p(\dualLat{L})$. Note that $f \in \ell^{\infty}(\dualLat{L})$. Then we find that
	\[
	\normIn{f}{\ell^q(\dualLat{L})}^q = \sum_{\kappa \in \dualLat{L}} \abs{f(\kappa)}^q
	= \sum_{\kappa \in \dualLat{L}} \abs{f(\kappa)}^p \abs{f(\kappa)}^{q-p}
	\leq \normIn{f}{\ell^{\infty}(\dualLat{L})}^{q-p} \normIn{f}{\ell^p(\dualLat{L})}^p,
	\]
	which proves our claim.
\end{proof}

\begin{remark}\label{remark:lpEmbedding}
	When we want to prove that a function $f : \dualLat{L} \to \bC$ belongs to $\ell^{\beta}(\dualLat{L})$ for all $\beta$ in a certain interval, it is sufficient to prove that it belongs to the left endpoint of the interval if the interval contains such a point, or else for $\beta$ arbitrarily close to the left endpoint of the interval, as \cref{prop:lpEmbedding} ensures that $f$ belongs to $\ell^{\beta}(\dualLat{L})$ for all larger $\beta$.
\end{remark}

\subsection{Additive case}

Let us introduce the additive H\"older-Lipschitz spaces $\HoldLipSpaceFD{\alpha}{p}$. After deriving a suitable Duren-type lemma for lattices, we then state and prove the first Titchmarsh theorem for these function spaces. Note that our proof is similar to the one of \cref{th:firstTitchRdAdd}.

\begin{definition}
	Let \(\alpha \in \bR\) with \(0 < \alpha \leq 1\), and let \(1 < p \leq \infty\). We define the (additive) \emph{H\"older-Lipschitz space} \(\HoldLipSpaceFD{\alpha}{p}\) by
	\[\HoldLipSpaceFD{\alpha}{p} := \left\{f \in L^{p}(\Omega) : \normIn{f(\cdot+h)-f(\cdot)}{L^{p}(\Omega)} = O\left(\abs{h}^{\alpha}\right) \text{ as } h \to 0\right\}.\]
\end{definition}

\begin{remark}
	Note that in the previous definition \(0\) may not be in \(\Omega\), but then we mean that \(h\) approaches the representative of the coset \(0 + L\) in \(\Omega\). Alternatively, if we extend the function periodically, then it makes sense to consider \(h \to 0\).
\end{remark}

\begin{remark}
	Some authors prefer another definition of the H\"older-Lipschitz spaces, namely for the parameters \(0 < \alpha_{1}, \dots, \alpha_{d} \leq 1\) and \(1 < p \leq \infty\) the (additive) H\"older-Lipschitz space is given by
	\[\HoldLipSpaceFD{\alpha_{1},\dots,\alpha_{d}}{p} := \left\{f \in L^{p}(\Omega) : \normIn{f(\cdot+h)-f(\cdot)}{L^{p}(\Omega)} = O\left(\sum_{j = 1}^{d} \abs{h_{j}}^{\alpha_{j}}\right) \text{ as } h \to 0\right\}.\]
	Remark that indeed \(\HoldLipSpaceFD{\alpha_{1},\dots,\alpha_{d}}{p} \subseteq \HoldLipSpaceFD{\alpha_{\textnormal{min}}}{p}\) with \(\alpha_{\textnormal{min}} := \min_{1 \leq j \leq d} \alpha_{j}\). Consequently, all of our results are also valid for these other H\"older-Lipschitz spaces. Actually, we would not be able to derive stronger results for the additive H\"older-Lipschitz spaces with the methods presented in this paper as suggested by \cref{th:firstTitchRdAdd}, so this inclusion is rather a simplification of notation than a restriction. On the other hand, for every \(0 < \alpha \leq 1\) we have \(\HoldLipSpaceFD{\alpha}{p} \subseteq \HoldLipSpaceFD{\alpha, \dots, \alpha}{p}\). In fact, \(\HoldLipSpaceFD{\alpha}{p} = \HoldLipSpaceFD{\alpha, \dots, \alpha}{p}\).
\end{remark}

We now derive a Duren-type lemma for lattices suited to additive H\"older-Lipschitz functions from the one-dimensional version of \cref{lemma:Duren}.

\begin{lemma}[Duren-type lemma for lattices, additive case]\label{lemma:DurenLatAdd}
	Let $d \in \bN \setminus \{0\}$, and let $L$ be a lattice in $\bR^d$. Consider a non-negative sequence $c_{\lambda} \geq 0$ with $\lambda \in L$, and let $0 < a < b$. Then the estimate
	\begin{equation}\label{eq:DurenLattice1}
		\sum_{\abs{\lambda} \leq N} \abs{\lambda}^b c_{\lambda} = O(N^a)
	\end{equation}
	is equivalent to
	\begin{equation}\label{eq:DurenLattice2}
		\sum_{\abs{\lambda} > N} c_{\lambda} = O(N^{a-b}).
	\end{equation}
\end{lemma}

\begin{proof}
	Firstly, suppose that \eqref{eq:DurenLattice1} holds. Let $N \in \bN \setminus \{0\}$. In this case we have that
	\[\sum_{n=1}^N n^b \left( \sum_{n < \abs{\lambda} \leq n+1} c_{\lambda} \right)
	< \sum_{n=0}^N \sum_{n < \abs{\lambda} \leq n+1} \abs{\lambda}^b c_{\lambda}
	\leq \sum_{\abs{\lambda} \leq N+1} \abs{\lambda}^b c_{\lambda}
	= O(N^a).\]
	It follows from Duren's lemma \ref{lemma:Duren} that
	\[\sum_{\abs{\lambda} > N} c_{\lambda}
	= \sum_{n \geq N} \left( \sum_{n < \abs{\lambda} \leq n+1} c_{\lambda} \right)
	= O(N^{a-b}).\]
	
	Conversely, if \eqref{eq:DurenLattice2} holds, we have that
	\[\sum_{n > N} \left( \sum_{n-1 < \abs{\lambda} \leq n} c_{\lambda} \right)
	= \sum_{\abs{\lambda} > N} c_{\lambda}
	= O(N^{a-b})\]
	so that Duren's lemma \ref{lemma:Duren} implies that
	\[\sum_{\abs{\lambda} \leq N} \abs{\lambda}^{b} c_{\lambda}
	\leq \sum_{n=1}^{N} n^{b} \left( \sum_{n-1 < \abs{\lambda} \leq n} c_{\lambda} \right)
	= O(N^a),\]
	proving the lemma.
\end{proof}

With the help of this version of Duren's lemma for lattices, we are now prepared to prove the first Titchmarsh theorem for the additive H\"older-Lipschitz functions.

\begin{theorem}\label{th:firstTitchmarsh}
	If \(f \in \HoldLipSpaceFD{\alpha}{p}\) for some \(0 < \alpha \leq 1\) and \(1 < p \leq 2\), then $\widehat{f} \in \ell^{\gamma}(\dualLat{L})$ for
	\[\frac{p}{p + \frac{\alpha}{d} p - 1} < \gamma \leq \infty.\]
\end{theorem}

\begin{proof}
	Let \(f \in \HoldLipSpaceFD{\alpha}{p}\). Note that for every \(h \in \bR^{d}\) we have
	\begin{align*}
		\FT{f(\cdot+h) - f(\cdot)}(\kappa)
		= \left(e^{2 \pi i \kappa \cdot h} - 1\right) \FT{f}(\kappa)
		&= e^{i \pi \kappa \cdot h} \left( e^{i \pi \kappa \cdot h} - e^{- i \pi \kappa \cdot h} \right) \FT{f}(\kappa) \\
		&= 2 i \sin\left(\pi \kappa \cdot h\right) \, e^{i \pi \kappa \cdot h} \FT{f}(\kappa).
	\end{align*}
	Using the Hausdorff-Young inequality \eqref{ineq:HY} and the assumption that $f \in \HoldLipSpaceFD{\alpha}{p}$, this leads to the following estimate:
	\begin{align*}
		\sum_{\abs{\kappa} \leq \frac{1}{2\abs{h}}} 2^{p'} |\sin(\pi \kappa \cdot h)|^{p'} |\FT{f}(\kappa)|^{p'}
		&\leq \sum_{\kappa \in \dualLat{L}}  2^{p'} |\sin(\pi \kappa \cdot h)|^{p'} |\FT{f}(\kappa)|^{p'} \\
		&\leq \normIn{f(\cdot + h) - f(\cdot)}{L^p(\Omega)}^{p'} \\
		&= O(\abs{h}^{\alpha p'}).
	\end{align*}
	Now we need to estimate $|\sin{(\pi \kappa \cdot h)}|$ from below. Notice that the condition $\abs{\kappa} \leq \frac{1}{2\abs{h}}$ allows us to apply Jordan's inequality \eqref{ineq:JordanIneq}, which gives the inequality
	$$2 |\kappa \cdot h| \leq |\sin{(\pi \kappa \cdot h)}| \leq  \pi |\kappa \cdot h|.$$
	Let $N \in \bN \setminus \{0\}$ and define $h_i := \frac{1}{2N} e_i$ for $1 \leq i \leq d$.
	Then it is clear that 
	\[
	\sum_{i= 1}^d |\sin(\pi \kappa \cdot h_i)|^{p'}
	\geq \left(\frac{2}{2N}\right)^{p'} \sum_{i=1}^{d} |\kappa \cdot e_i|^{p'}
	= \left(\frac{1}{N}\right)^{p'} |\kappa|_{p'}^{p'}
	\geq \left(\frac{1}{N}\right)^{p'} d^{\frac{1}{2} - \frac{1}{p'}} |\kappa|^{p'},
	\]
	so that the following estimate holds true:
	\begin{equation}\label{eq:FTTA2}
		\sum_{\abs{\kappa} \leq N} |\kappa|^{p'} |\FT{f}(\kappa)|^{p'} \leq N^{p'} d^{\frac{1}{p'} - \frac{1}{2}} \sum_{i=1}^d \sum_{\abs{\kappa}
			\leq \frac{1}{2\abs{h_i}}} |\sin (\pi \kappa \cdot h_i)|^{p'} |\FT{f}(\kappa)|^{p'}
		= O\left(N^{(1-\alpha)p'}\right).
	\end{equation}
	Using H\"older's inequality with exponents $\frac{p'}{\gamma}$ and $1/(1 - \frac{\gamma}{p'})$ where $1 \leq \gamma < p'$, and the estimate
	\[\sum_{\abs{\kappa} \leq N} 1 = O\left(N^d\right),\]
	we get
	\begin{equation}\label{eq:FTTA1}
		\sum_{\abs{\kappa} \leq N} \abs{\kappa}^{\gamma} |\FT{f}(\kappa)|^{\gamma}
		\leq \left( \sum_{\abs{\kappa} \leq N} \abs{\kappa}^{p'} |\FT{f}(\kappa)|^{p'} \right)^{\frac{\gamma}{p'}} \left( \sum_{\abs{\kappa} \leq N} 1 \right)^{1 - \frac{\gamma}{p'}}
		= O\left( N^{d + \gamma \left( 1 - \alpha - \frac{d}{p'} \right)} \right).
	\end{equation}
	Apply Duren's lemma \ref{lemma:DurenLatAdd} to \eqref{eq:FTTA1} in order to obtain that
	\begin{equation}\label{eq:FTTA3}
		\sum_{\abs{\kappa} > N} |\FT{f}(\kappa)|^{\gamma} = O\left(N^{d - \gamma \left(\alpha + \frac{d}{p'}\right)}\right),
	\end{equation}
	under the condition that $0 < d + \gamma \left(1 - \alpha - \frac{d}{p'}\right) < \gamma$. Since
	\[d + \gamma \left( 1 - \alpha - \frac{d}{p'} \right) = \gamma (1 - \alpha) + d \left( 1 - \frac{\gamma}{p'} \right) > 0,\]
	this condition amounts to
	\[\frac{d}{\alpha + \frac{d}{p'}} = \frac{d}{\alpha + d - \frac{d}{p}} = \frac{dp}{\alpha p + dp - d} = \frac{p}{p + \frac{\alpha}{d} p - 1} < \gamma.\]
	Note that in principle this result is only valid for $\gamma < p'$. The Hausdorff-Young inequality covers the single case $\gamma = p'$, but \cref{remark:lpEmbedding} extends the result to all $\gamma \geq p'$, concluding the proof.
\end{proof}

\begin{remark}\label{remark:HoldLipFirstTitchSharp}
	We cannot enlarge the range of \(\beta\) in \cref{th:firstTitchmarsh} to include \(\frac{p}{p + \frac{\alpha}{d} p - 1}\). There is already a one-dimensional counterexample, namely the function
	\[f(x) = \sum_{n=1}^{\infty} \frac{e^{i n \log n}}{n^{\frac{1}{2}+\alpha}}\]
	for \(0 < \alpha < 1\) is in \(\HoldLipSpace{\mathbb{T}}{\alpha}{2}\) but \(\FT{f} \not\in \ell^{\frac{2}{2\alpha + 1}}(\bZ)\). See \cite[page 243]{Zygmund} for details.
\end{remark}

Remark that the strategy that we employed in the proof of \ref{th:firstTitchmarsh} corresponds to the proof ideas from \cite[Theorem 84]{Titchmarsh} and \cite[Theorem 2.15]{LipThesis}. It is worth to note that a slightly different approach is being pursued in \cite[Theorem 3.2]{LipArticle} in the context of compact Lie groups to obtain an extra conclusion. However, the proof contains an inaccuracy. Namely, in the notation of that paper, the authors have found in \cite[Equation (3.6)]{LipArticle} that
\[\sum_{[\xi] \in \FT{G}, \langle \xi \rangle \leq N} d_{\xi}^2 \left( \frac{\langle \xi \rangle \lVert \FT{f}(\xi) \rVert_{\textnormal{HS}}}{\sqrt{d_{\xi}}} \right)^{\beta} = O\left( N^{(1-\alpha) \beta + n \left( 1 - \frac{\beta}{q} \right)} \right)\]
under the conditions that $n \in \bN \setminus \{0\}$, $0 < \alpha \leq 1$, $1 < p \leq 2$ with $q \in \bR$ such that $\frac{1}{p} + \frac{1}{q} = 1$, and $1 \leq \beta \leq q$.
Since the Fourier multiplier $(I - \cL_{G})^{\frac{1}{2}}$ has symbol $\langle \xi \rangle$, they derived that $\FT{(I - \cL_{G})^{\frac{1}{2}} f} \in \ell^{\beta}(\FT{G})$ if $(1 - \alpha) \beta + n \left( 1 - \frac{\beta}{q} \right) \leq 0$, but the conditions under which \cite[Equation (3.6)]{LipArticle} were derived imply that $(1 - \alpha) \beta + n \left( 1 - \frac{\beta}{q} \right) \geq 0$. Moreover, similar to the inaccuracy in the proof of the original first Titchmarsh theorem, the sum $\sum_{[\xi] \in \FT{G}, \langle \xi \rangle \leq N}$ is considered instead of the series $\sum_{[\xi] \in \FT{G}, \langle \xi \rangle > N}$.

Since this approach is interesting by itself and derives an extra result, we will discuss it in our context of fundamental domains of lattices. We start with fixing the inaccuracy for the range of $\beta$ in \cite[Theorem 3.2]{LipArticle}.

\begin{theorem}\label{th:HLSL}
	If $f \in \HoldLipSpaceFD{\alpha}{p}$ for some $0 < \alpha \leq 1$ and $1 < p \leq 2$, then $\FT{(I - \Delta_{\Omega})^{\frac{\delta}{2}} f} \in \ell^{\gamma}(\dualLat{L})$ for $\max\{\alpha - \frac{d}{p}, 0\} \leq \delta < \alpha$ and $\frac{dp'}{(\alpha - \delta) p' + d} < \gamma \leq \infty$.
\end{theorem}

\begin{proof}
	Let $1 \leq \gamma \leq p'$ and $\delta \geq 0$. Applying H\"older's inequality with exponents $\frac{p'}{\gamma}$ and $\frac{1}{1-\frac{\gamma}{p'}}$, we find that
	\[
	\sum_{\abs{\kappa} \leq N} \abs{\kappa}^{\delta \gamma} \abs{\kappa}^{\gamma} |\FT{f}(\kappa)|^{\gamma}
	\leq \left(\sum_{\abs{\kappa} \leq N} \abs{\kappa}^{p'} |\FT{f}(\kappa)|^{p'}\right)^{\frac{\gamma}{p'}} \left(\sum_{\abs{\kappa} \leq N} \abs{\kappa}^{\frac{\delta \gamma}{1 - \frac{\gamma}{p'}}}\right)^{1 - \frac{\gamma}{p'}}.
	\]
	Note that
	\[\left(\sum_{\abs{\kappa} \leq N} \abs{\kappa}^{\frac{\delta \gamma}{1 - \frac{\gamma}{p'}}}\right)^{1 - \frac{\gamma}{p'}}
	\leq \left( N^{\frac{\delta \gamma}{1 - \frac{\gamma}{p'}}} \sum_{\abs{\kappa} \leq N}  1 \right)^{1 - \frac{\gamma}{p'}}
	= N^{\delta \gamma + d\left(1 - \frac{\gamma}{p'}\right)}\]
	because $\frac{\delta \gamma}{1 - \frac{\gamma}{p'}} > 0$, which is guaranteed to be the case as $1 \leq \gamma \leq p'$ and $\delta > 0$. Hence, we find using \eqref{eq:FTTA2} that
	\[\sum_{\abs{\kappa} \leq N} \abs{\kappa}^{\delta \gamma} \abs{\kappa}^{\gamma} |\FT{f}(\kappa)|^{\gamma}
	= O\left(N^{(1 - \alpha)\gamma + \delta \gamma + d \left(1 - \frac{\gamma}{p'}\right)}\right).\]
	We can now apply Duren's lemma \ref{lemma:DurenLatAdd} to obtain that
	\[\sum_{\abs{\kappa} > N} \abs{\kappa}^{\delta \gamma} |\FT{f}(\kappa)|^{\gamma}
	= O\left(N^{(\delta - \alpha) \gamma + d \left(1 - \frac{\gamma}{p'}\right)}\right) \]
	under the condition that $0 < (1 - \alpha) \gamma + \delta \gamma + d \left( 1 - \frac{\gamma}{p'} \right) < \gamma$. Note that the left inequality is satisfied, and the right one can be rewritten as
	\begin{equation}\label{ineq:LI1}
		(\delta - \alpha) \gamma + d \left( 1 - \frac{\gamma}{p'}\right) < 0.
	\end{equation}
	Note that \eqref{ineq:LI1} can only be satisfied if $\delta < \alpha$. We can rewrite \eqref{ineq:LI1} as
	\[\gamma > \frac{dp'}{(\alpha - \delta)p' + d}.\]
	Since $1 \leq \gamma \leq p'$, the latter inequality can only be satisfied if $1 \leq \frac{dp'}{(\alpha - \delta)p' + d} < p'$. We see that
	\[\frac{dp'}{(\alpha - \delta)p' + d} = \frac{d}{d + (\alpha - \delta)p'} p' < p'\]
	is true as $(\alpha - \delta)p' > 0$. A direct calculation shows that the other inequality is valid when
	\[\delta \geq \alpha - \frac{d}{p}.\]
	Taking the restriction $\delta \geq 0$ and \cref{remark:lpEmbedding} into account, this proves the result.
\end{proof}

The idea in \cite[Theorem 3.2]{LipArticle} is to derive from $\FT{(I - \Delta_{\Omega})^{\frac{\delta}{2}} f} \in \ell^{\gamma}(\dualLat{L})$ for a suitable range of $\gamma$ and $\delta$, a range for $\beta$ for which $\FT{f} \in \ell^{\beta}(\dualLat{L})$. This step is captured in the following lemma.

\begin{lemma}\label{lemma:Symbol}
	Let $\sigma : \dualLat{L} \to \bC$ be a function. Let $1 \leq r < \infty$ and $\gamma > 0$ be positive real numbers. If $\langle \kappa \rangle^{\gamma} \sigma(\kappa) \in \ell^{r}(\dualLat{L})$ then $\sigma \in \ell^\beta(\dualLat{L})$ for all $ \frac{r d}{d +\gamma r} < \beta \leq \infty$. Here, we are using the notation $\langle \kappa \rangle : = (1 + \abs{\kappa}^{2})^{1/2}$.
\end{lemma}

\begin{proof}
	For $\beta < r$ we find using H\"older's inequality that
	\[\normIn{\sigma}{\ell^{\beta}(\dualLat{L})}^{\beta}
	= \sum_{\kappa \in \dualLat{L}} \langle \kappa \rangle^{\gamma \beta} \abs{\sigma(\kappa)}^{\beta} \langle \kappa \rangle^{-\gamma \beta}
	\leq \left( \sum_{\kappa \in \dualLat{L}} \langle \kappa \rangle^{\gamma r} \abs{\sigma(\kappa)}^r \right)^{\frac{\beta}{r}} \left( \sum_{\kappa \in \dualLat{L}} \langle \kappa \rangle^{- \frac{\gamma \beta r}{r - \beta}}\right)^{1 - \frac{\beta}{r}}.\]
	We observe that
	\[\sum_{\kappa \in \dualLat{L}} \langle \kappa \rangle^{- \frac{\gamma \beta r}{r - \beta}} = \sum_{\kappa \in \dualLat{L}} \frac{1}{(1 + \abs{\kappa}^2)^{\frac{1}{2} \frac{\gamma \beta r}{r - \beta}}}\]
	converges if $\frac{\gamma \beta r}{r - \beta} > d$, or in other words when $\beta > \frac{rd}{\gamma r + d}$, which proves the lemma because of \cref{remark:lpEmbedding}.
\end{proof}

Now, we can derive the first Titchmarsh theorem by applying \cref{lemma:Symbol} to \cref{th:HLSL}. Namely, since $\FT{(I - \Delta_{\Omega})^{\frac{\delta}{2}} f} \in \ell^{\gamma}(\dualLat{L})$ for $\max\{\alpha - \frac{d}{p}, 0\} \leq \delta < \alpha$ and $\frac{dp'}{(\alpha - \delta) p' + d} < \gamma \leq \infty$, we find formally by choosing $\gamma \to \alpha$ that $r = \frac{dp'}{(\alpha - \gamma) p' + d} \to p'$ so that $\FT{f} \in \ell^{\beta}(\dualLat{L})$ for $\beta > \frac{rd}{d + \gamma r} \to \frac{dp'}{d + \alpha p'} = \frac{p}{p + \frac{\alpha}{d} p - 1}$. The reader should have no difficulty in making this reasoning precise.

We proceed to a discussion of the second Titchmarsh theorem for the additive H\"older-Lipschitz spaces. Our proof is based on the one of \cite[Theorem 2.17]{LipThesis}, which treats the two-dimensional case.

\begin{theorem}\label{th:secondTitchmarsh}
	Let \(0 < \alpha < 1\) and \(f \in L^{2}(\Omega)\). Then \(f \in \HoldLipSpaceFD{\alpha}{2}\) if and only if
	\begin{equation}\label{eq:secondTitchmarsh}
		\sum_{\abs{\kappa} > N} |\FT{f}(\kappa)|^{2} = O\left(N^{-2\alpha}\right).
	\end{equation}
\end{theorem}

\begin{proof}
	First assume that \(f \in \HoldLipSpace{\Omega}{\alpha}{2}\). Following the arguments in the proof of \cref{th:firstTitchmarsh}, note that \eqref{eq:FTTA3} is valid for $\gamma = p'$ if $0 < \alpha < 1$. Hence, for $p = p' = 2$ we get that
	\[\sum_{\abs{\kappa} > N} |\FT{f}(\kappa)|^2 = O\left( N^{-2\alpha} \right).\]
	
	Conversely, assume that \eqref{eq:secondTitchmarsh} holds. An application of Duren's lemma \ref{lemma:DurenLatAdd} to \eqref{eq:secondTitchmarsh} gives that
	\[\sum_{\abs{\kappa} \leq N} \abs{\kappa}^2 |\FT{f}(\kappa)|^2 = O\left( N^{2-2\alpha} \right)\]
	under the condition that $0 < \alpha < 1$, which is one of the assumptions in the theorem. Consider $h \in \bR^d$ with $\abs{h}$ sufficiently small, and set $N_h := \floor{\frac{1}{\abs{h}}}$. Hence, via the inequality \(\abs{\sin t} \leq \abs{t}\) for all \(t \in \bR\), we obtain that
	\begin{align*}
		\sum_{\kappa \in \dualLat{L}} \sin^{2}(\pi \kappa \cdot h) |\FT{f}(\kappa)|^{2}
		&\leq \pi^{2} \abs{h}^{2} \sum_{\abs{\kappa} \leq N_h} \abs{\kappa}^2 |\FT{f}(\kappa)|^{2} + \sum_{\abs{\kappa} > N_h} |\FT{f}(\kappa)|^{2} \\
		&\lesssim \abs{h}^2 N_h^{2-2\alpha} + N_h^{-2\alpha} \\
		&\lesssim \abs{h}^{2\alpha},
	\end{align*}
	where we used that $\frac{1}{2 \abs{h}} \leq N_h \leq \frac{1}{\abs{h}}$ if $N_h \geq 1$; that is, if $\abs{h}$ is sufficiently small. Noting that Plancherel's theorem \eqref{eq:Plancherel} gives the relation
	\[\normIn{f(\cdot+h)-f(\cdot)}{L^{2}(\Omega)}^{2} = \sum_{\kappa \in \dualLat{L}} 4 \sin^{2}(\pi \kappa \cdot h) \, |\FT{f}(\kappa)|^{2},\]
	this completes the proof.
\end{proof}

\subsection{Multiplicative case}

Remark that the conclusion of the first Titchmarsh theorem for additive H\"older-Lipschitz spaces weakens as the dimension \(d\) increases. This feature does not seem to be inherent for every function space satisfying a H\"older-Lipschitz condition, since we will show that our first Titchmarsh theorem for the multiplicative H\"older-Lipschitz spaces is independent of the dimension. The crucial ingredient in the proof of this result is the following generalization of Duren's lemma.

\begin{lemma}[Duren-type lemma for lattices, multiplicative case]\label{lemma:DurenLatMult}
	Let $d \in \bN \setminus \{0\}$, and let $L$ be a lattice in $\bR^d$. Consider a non-negative sequence $c_{\lambda} \geq 0$ with $\lambda \in L$, and let $0 < a_n < b_n$ for every $1 \leq n \leq d$. Then all estimates of the following form are equivalent:
	\begin{equation}\label{eq:DurenLatticeMultEstimate}
		\sum_{\lambda_{1} \in I_{\varepsilon_{1}(N_{1})}} \dots \sum_{\lambda_{d} \in I_{\varepsilon_{d}(N_{d})}} \big|\lambda_{1}\big|^{\varepsilon_{1} b_1} \dots \big|\lambda_{d}\big|^{\varepsilon_{d} b_d} c_{\lambda}
		= O\left(N_{1}^{a_1 - (1 - \varepsilon_{1}) b_1} \dots N_{d}^{a_d - (1 - \varepsilon_{d}) b_d}\right),
	\end{equation}
	where \(\varepsilon_{1}, \dots, \varepsilon_{d} \in \{0,1\}\), \(I_{1}(N) := [-N,N]\) and \(I_{0}(N) := \bR \setminus [-N,N]\) for \(N \in \bN\).
\end{lemma}

\begin{proof}
	It suffices to prove this result sum by sum, and without loss of generality we can prove it for the first sum, i.e.,
	\begin{multline}\label{eq:DurenLatMult1}
		\sum_{\abs{\lambda_1} \leq N_1} \sum_{\lambda_2 \in I_{\varepsilon_2}(N_2)} \dots \sum_{\lambda_d \in I_{\varepsilon_{d}}(N_d)} \abs{\lambda_1}^{b_1} \abs{\lambda_2}^{\varepsilon_2 b_2} \dots \abs{\lambda_d}^{\varepsilon_d b_d} c_{\lambda} \\
		= O\left( N_{1}^{a_1} N_{2}^{a_2 - (1 - \varepsilon_2) b_2} \dots N_{d}^{a_d - (1 - \varepsilon_d) b_d} \right)
	\end{multline}
	is equivalent to
	\begin{multline}\label{eq:DurenLatMult2}
		\sum_{\abs{\lambda_1} > N_1} \sum_{\lambda_2 \in I_{\varepsilon_2}(N_2)} \dots \sum_{\lambda_d \in I_{\varepsilon_{d}}(N_d)} \abs{\lambda_2}^{\varepsilon_2 b_2} \dots \abs{\lambda_d}^{\varepsilon_d b_d} c_{\lambda} \\
		= O\left( N_{1}^{a_1 - b_1} N_{2}^{a_2 - (1 - \varepsilon_2) b_2} \dots N_{d}^{a_d - (1 - \varepsilon_d) b_d} \right).
	\end{multline}
	For ease of notation, set $C_{\lambda_1, N_2, \dots, N_d} := \sum_{\lambda_2 \in I_{\varepsilon_2}(N_2)} \dots \sum_{\lambda_d \in I_{\varepsilon_{d}}(N_d)} \abs{\lambda_2}^{\varepsilon_2 b_2} \dots \abs{\lambda_d}^{\varepsilon_d b_d} c_{\lambda}$.
	
	Assume that \eqref{eq:DurenLatMult1} holds. In this case we have
	\begin{align*}
		\sum_{n_{1} = 1}^{N_1} n_{1}^{b_1} \left( \sum_{n_1 < \abs{\lambda_1} \leq n_{1} + 1} C_{\lambda_1, N_2, \dots, N_d} \right)
		&\leq \sum_{n_{1} = 0}^{N_1} \sum_{n_1 < \abs{\lambda_1} \leq n_1 + 1} \abs{\lambda_1}^{b_1} C_{\lambda_1, N_2, \dots, N_d} \\
		&= \sum_{\abs{\lambda_{1}} \leq N_1} \abs{\lambda_1}^{b_1} C_{\lambda_1, N_2, \dots, N_d} \\
		&= O\left( N_{1}^{a_1} N_{2}^{a_2 - (1 - \varepsilon_2) b_2} \dots N_{d}^{a_d - (1 - \varepsilon_d) b_d} \right).
	\end{align*}
	It follows from Duren's lemma \ref{lemma:Duren} that
	\begin{align*}
		\sum_{\abs{\lambda_1} > N_1} C_{\lambda_1, N_2, \dots, N_d}
		&= \sum_{n_1 \geq N_1} \sum_{n_1 < \abs{\lambda_1} \leq n_{1} + 1} C_{\lambda_1, N_2, \dots, N_d} \\
		&= O\left( N_{1}^{a_1 - b_1} N_{2}^{a_2 - (1 - \varepsilon_2) b_2} \dots N_{d}^{a_d - (1 - \varepsilon_d) b_d} \right).
	\end{align*}
	
	Conversely, assume that \eqref{eq:DurenLatMult2} holds. Then
	\begin{align*}
		\sum_{n_1 > N_1} \sum_{n_{1} - 1 < \abs{\lambda_1} \leq n_1} C_{\lambda_1, N_2, \dots, N_d} &= \sum_{\abs{\lambda_1} > N_1} C_{\lambda, N_2, \dots, N_d} \\
		&= O\left( N_{1}^{a_1 - b_1} N_{2}^{a_2 - (1 - \varepsilon_2) b_2} \dots N_{d}^{a_d - (1 - \varepsilon_d) b_d} \right).
	\end{align*}
	We get from Duren's lemma \ref{lemma:Duren} that
	\begin{align*}
		\sum_{\abs{\lambda_1} \leq N_1} \abs{\lambda_1}^{b_1} C_{\lambda_1, N_2, \dots, N_d}
		&\leq \sum_{n_1 = 1}^{N_1} \sum_{n_{1} - 1 < \abs{\lambda_1} \leq n_1} n_{1}^{b_1} C_{\lambda_1, N_2, \dots, N_d} \\
		&= O\left( N_{1}^{a_1} N_{2}^{a_2 - (1 - \varepsilon_2) b_2} \dots N_{d}^{a_d - (1 - \varepsilon_d) b_d} \right),
	\end{align*}
	which proves the lemma.
\end{proof}

Let us now introduce some operators that represent coordinate-wise translations rather a global translation as in the additive case.

\begin{definition}\label{def:diffOpMult}
	We define for any \(1 \leq j \leq d\) and \(h_{j} \in \bR\) the difference operator \(\diffOpForw_{j}^{h_{j}}\) acting on a function \(f \in L^{p}(\Omega)\) for some \(1 \leq p \leq \infty\) by
	\[\diffOpForw_{j}^{h_{j}} f(x) := f(x + h_{j} e_{j}) - f(x),\]
	where \(e_{j}\) is the \(j\)-th element of the standard basis of \(\bR^{d}\).
\end{definition}

These coordinate-wise difference operators enable us to define the multiplicative H\"older-Lipschitz spaces $\HoldLipSpaceMultFD{\alpha}{p}$.

\begin{definition}
	Let \(\alpha \in \bR\) with \(0 < \alpha \leq 1\), and let \(1 < p \leq \infty\). We define the (multiplicative) \emph{H\"older-Lipschitz space} \(\HoldLipSpaceMultFD{\alpha}{p}\) by
	\[\HoldLipSpaceMultFD{\alpha}{p} := \left\{f \in L^{p}(\Omega) : \normIn{\diffOpForw_{1}^{h_{1}} \dots \diffOpForw_{d}^{h_{d}} f}{L^{p}(\Omega)} = O\left(\abs{h_{1}}^{\alpha} \dots \abs{h_{d}}^{\alpha}\right) \text{ as } h \to 0\right\}.\]
\end{definition}

\begin{remark}\label{rem:HoldLipMultAlt}
	As in the case of additive H\"older-Lipschitz spaces, we can consider an alternative multiplicative H\"older-Lipschitz space that, for any \(0 < \alpha_{1}, \dots, \alpha_{d} \leq 1\), is given by
	\[\HoldLipSpaceMultFD{\alpha_{1}, \dots, \alpha_{d}}{p} := \left\{f \in L^{p}(\Omega) : \normIn{\diffOpForw_{1}^{h_{1}} \dots \diffOpForw_{d}^{h_{d}} f}{L^{p}(\Omega)} = O\left(\abs{h_{1}}^{\alpha_{1}} \dots \abs{h_{d}}^{\alpha_{d}}\right) \text{ as } h \to 0\right\}.\]
	Remark that \(\HoldLipSpaceMultFD{\alpha_{1}, \dots, \alpha_{d}}{p} \subseteq \HoldLipSpaceMultFD{\alpha_{\text{min}}}{p}\) with \(\alpha_{\text{min}} := \min_{1 \leq j \leq d} \alpha_{j}\), because \(\abs{h}^{\alpha} \leq \abs{h}^{\beta}\) if \(0 < \beta \leq \alpha \leq 1\) for \(\abs{h} < 1\). Obviously, we also have \(\HoldLipSpaceMultFD{\alpha}{p} = \HoldLipSpaceMultFD{\alpha, \dots, \alpha}{p}\).
\end{remark}

We are now ready to prove the first Titchmarsh theorem for multiplicative H\"older-Lipschitz spaces. The proof ideas follow the same lines as those for \cref{th:firstTitchMultRd}.

\begin{theorem}\label{th:HoldLipMultFirstTitch}
	If $f \in \HoldLipSpaceMultFD{\alpha}{p}$ for some $0 < \alpha \leq 1$ and $1 < p \leq 2$, then
	$f \in \ell^{\gamma}(\dualLat{L})$ for
	\[\frac{p}{p + \alpha p - 1} < \gamma \leq \infty.\]
\end{theorem}

\begin{proof}
	Let \(f \in \HoldLipSpaceMultFD{\alpha}{p}\), and let \(N_{1}, \dots, N_{d} \in \bN\) be sufficiently large. First we compute for \(\kappa \in \dualLat{L}\) that
	\begin{align}\label{eq:FourierHoldLipMult}
		\FT{\diffOpForw_{1}^{\frac{1}{2N_{1}}} \dots \diffOpForw_{d}^{\frac{1}{2N_{d}}} f}(\kappa)
		&= \left(e^{2 \pi i \frac{\kappa_{1}}{2N_{1}}} - 1\right) \dots \left(e^{2 \pi i \frac{\kappa_{d}}{2N_{d}}} - 1\right) \FT{f}(\kappa) \nonumber \\
		&= 2 i \sin\left(\frac{\pi \kappa_{1}}{2N_{1}}\right) \, e^{i \pi \frac{\kappa_{1}}{2N_{1}}} \dots 2 i \sin\left(\frac{\pi \kappa_{d}}{2N_{d}}\right) \, e^{i \pi \frac{\kappa_{d}}{2N_{d}}} \FT{f}(\kappa).
	\end{align}
	By applying Jordan's inequality \eqref{ineq:JordanIneq} to \eqref{eq:FourierHoldLipMult} we obtain that
	\begin{align*}
		2^{d p'} \sum_{|\kappa_{1}| \leq N_{1}} \dots \sum_{|\kappa_{d}| \leq N_{d}} \abs{\frac{\kappa_{1}}{N_{1}}}^{p'} \dots \abs{\frac{\kappa_{d}}{N_{d}}}^{p'} |\FT{f}(\kappa)|^{p'}
		&\leq \sum_{|\kappa_{1}| \leq N_{1}} \dots \sum_{|\kappa_{d}| \leq N_{d}} \abs{\FT{\diffOpForw_{1}^{\frac{1}{2N_{1}}} \dots \diffOpForw_{d}^{\frac{1}{2N_{d}}} f}(\kappa)}^{p'} \\
		&\leq \normIn{\diffOpForw_{1}^{\frac{1}{2N_{1}}} \dots \diffOpForw_{d}^{\frac{1}{2N_{d}}} f}{L^{p}(\Omega)}^{p'} \\
		&= O\left(N_{1}^{-\alpha p'} \dots N_{d}^{-\alpha p'}\right),
	\end{align*}
	where we applied the Hausdorff-Young inequality \eqref{ineq:HY}.
	Thus, we have found the estimate
	\begin{equation}\label{eq:HLMFT1}
		\sum_{|\kappa_{1}| \leq N_{1}} \dots \sum_{|\kappa_{d}| \leq N_{d}} \big|\kappa_{1}\big|^{p'} \dots \big|\kappa_{d}\big|^{p'} |\FT{f}(\kappa)|^{p'} = O\left(N_{1}^{(1-\alpha)p'} \dots N_{d}^{(1-\alpha)p'}\right).
	\end{equation}
	Using H\"older's inequality with exponents $\frac{p'}{\gamma}$ and $1/(1 - \frac{\gamma}{p'})$ where $1 \leq \gamma < p'$, and the estimate
	\[\sum_{|\kappa_{1}| \leq N_{1}} \dots \sum_{|\kappa_{d}| \leq N_{d}} 1 = O(N_1 \dots N_d)\]
	we get
	\begin{multline*}
		\sum_{\abs{\kappa_{1}} \leq N_1} \dots \sum_{\abs{\kappa_{d}} \leq N_d} \big|\kappa_{1}\big|^{\gamma} \dots \big|\kappa_{d}\big|^{\gamma} |\FT{f}(\kappa)|^{\gamma} \\
		\leq \left( \sum_{|\kappa_{1}| \leq N_{1}} \dots \sum_{|\kappa_{d}| \leq N_{d}} \big|\kappa_{1}\big|^{p'} \dots \big|\kappa_{d}\big|^{p'} |\FT{f}(\kappa)|^{p'} \right)^{\frac{\gamma}{p'}} \left( \sum_{|\kappa_{1}| \leq N_{1}} \dots \sum_{|\kappa_{d}| \leq N_{d}} 1 \right)^{1 - \frac{\gamma}{p'}} \\
		= O\left(N_1^{(1 - \alpha) \gamma + 1 - \frac{\gamma}{p'}} \dots N_d^{(1 - \alpha) \gamma + 1 - \frac{\gamma}{p'}}\right).
	\end{multline*}
	Apply Duren's lemma \ref{lemma:DurenLatMult} to obtain that
	\[
		\sum_{|\kappa_{1}| > N_1} \dots \sum_{|\kappa_{d}| > N_{d}} |\FT{f}(\kappa)|^{\gamma} = O\left( N_{1}^{1 - \alpha \gamma - \frac{\gamma}{p'}} \dots N_{d}^{1 - \alpha \gamma - \frac{\gamma}{p'}} \right)
	\]
	under the condition that $0 < (1 - \alpha) \gamma + 1 - \frac{\gamma}{p'} < \gamma$. Note that the left inequality is satisfied, and the right one is equivalent to
	\[\frac{1}{\alpha + \frac{1}{p'}} = \frac{1}{\alpha + 1 - \frac{1}{p}} = \frac{p}{p + \alpha p - 1} < \gamma.\]
	Keeping \cref{remark:lpEmbedding} in mind, this completes the proof.
\end{proof}

\begin{remark}
	Since additive and multiplicative H\"older-Lipschitz spaces coincide in dimension $1$, \cref{remark:HoldLipFirstTitchSharp} also shows that the result in \cref{th:HoldLipMultFirstTitch} is sharp.
\end{remark}

Next, we treat the second Titchmarsh theorem for multiplicative H\"older-Lipschitz functions. For full generality of this result we consider the multiplicative H\"older-Lipschitz spaces $\HoldLipSpaceMultFD{\alpha_{1}, \dots, \alpha_{d}}{p}$ from \cref{rem:HoldLipMultAlt}. The inspiration for our proof stems from the reasoning in the proof of \cite[Theorem 2.19]{LipThesis}.

\begin{theorem}\label{th:secondTitchMult}
    Let \(0 < \alpha_{1}, \dots, \alpha_{d} < 1\) and \(f \in L^{2}(\Omega)\). Then \(f \in \HoldLipSpaceMultFD{\alpha_{1}, \dots, \alpha_{d}}{2}\) if and only if any estimate of the following form holds:
    \begin{equation}\label{eq:secondTitchmarshMult}
        \sum_{\kappa_{1} \in I_{\varepsilon_{1}}(N_{1})} \dots \sum_{\kappa_{d} \in I_{\varepsilon_{d}}(N_{d})} \big|\kappa_{1}\big|^{2\varepsilon_{1}} \dots \big|\kappa_{d}\big|^{2\varepsilon_{d}} |\FT{f}(\kappa)|^{2} = O\left(N_{1}^{-2\alpha_{1} + 2\varepsilon_{1}} \dots N_{d}^{-2\alpha_{d} + 2\varepsilon_{d}}\right),
    \end{equation}
    where \(\varepsilon_{1}, \dots, \varepsilon_{d} \in \{0,1\}\), \(I_{1}(N) := [-N,N]\) and \(I_{0}(N) := \bR \setminus [-N,N]\) for \(N \in \bN\).
    
    In particular, we have that \(f \in \HoldLipSpaceMultFD{\alpha_{1}, \dots, \alpha_{d}}{2}\) if and only if
    \[\sum_{\abs{\kappa_{1}} > N_{1}} \dots \sum_{\abs{\kappa_{d}} > N_{d}} |\FT{f}(\kappa)|^{2} = O\left(N_{1}^{-2\alpha_{1}} \dots N_{d}^{-2\alpha_{d}}\right).\]
\end{theorem}

\begin{proof}
    Suppose first that \(f \in \HoldLipSpaceMultFD{\alpha_1, \dots, \alpha_d}{2}\). Following the steps in \cref{th:HoldLipMultFirstTitch}, we find that \eqref{eq:HLMFT1} for $p = p' = 2$ becomes
    \begin{equation}\label{ineq:secondTitchmarshMult1}
    	\sum_{|\kappa_{1}| \leq N_{1}} \dots \sum_{|\kappa_{d}| \leq N_{d}} \big|\kappa_{1}\big|^{2} \dots \big|\kappa_{d}\big|^{2} |\FT{f}(\kappa)|^{2}
    	= O\left(N_{1}^{2(1-\alpha_1)} \dots N_{d}^{2(1-\alpha_d)}\right).
    \end{equation}
    It follows now from Duren's lemma \ref{lemma:DurenLatMult} that \eqref{ineq:secondTitchmarshMult1} is equivalent with any estimate of the form \eqref{eq:secondTitchmarshMult}. Note that the requirement $0 < 2(1 - \alpha_j) < 2$ for every $1 \leq j \leq d$ for Duren's lemma is satisfied since we assumed that $0 < \alpha_{j} < 1$ for all $j$.
    
    Conversely, assume that any estimate of the form \eqref{eq:secondTitchmarshMult} holds. By Duren's lemma \ref{lemma:DurenLatMult}, all estimates of this form hold. Let $N_{1}, \dots, N_{d} \in \bN$ be sufficiently large. Plancherel's theorem \eqref{eq:Plancherel} gives
    \[\normIn{\diffOpForw_{1}^{\frac{1}{2N_1}} \dots \diffOpForw_{d}^{\frac{1}{2N_d}} f}{L^2(\Omega)}^2
    = 2^{2d} \sum_{\kappa \in \dualLat{L}} \sin^2\left(\frac{\pi \kappa_1}{2N_1}\right) \dots \sin^2\left(\frac{\pi \kappa_d}{2N_d}\right) |\FT{f}(\kappa)|^2.\]
    We observe that
    \begin{multline*}
    	\sum_{\kappa \in \dualLat{L}} \sin^{2}\left(\frac{\pi \kappa_{1}}{2N_{1}}\right) \dots \sin^{2}\left(\frac{\pi \kappa_{d}}{2N_{d}}\right) |\FT{f}(\kappa)|^{2} \\
    	= \sum_{\varepsilon \in \{0,1\}^{d}} \sum_{\kappa_{1} \in I_{\varepsilon_{1}}(N_{1})} \dots \sum_{\kappa_{d} \in I_{\varepsilon_{d}}(N_{d})} \sin^{2}\left(\frac{\pi \kappa_{1}}{2N_{1}}\right) \dots \sin^{2}\left(\frac{\pi \kappa_{d}}{2N_{d}}\right) |\FT{f}(\kappa)|^{2}.
    \end{multline*}
    Using Jordan's inequality \eqref{ineq:JordanIneq} when applicable, or otherwise the estimate \(\sin^{2}\left(\frac{\pi \kappa_{j}}{2N_{j}}\right) \leq 1\), we find
    \begin{multline*}
        \sum_{\varepsilon \in \{0,1\}^{d}} \sum_{\kappa_{1} \in I_{\varepsilon_{1}}(N_{1})} \dots \sum_{\kappa_{d} \in I_{\varepsilon_{d}}(N_{d})} \sin^{2}\left(\frac{\pi \kappa_{1}}{2N_{1}}\right) \dots \sin^{2}\left(\frac{\pi \kappa_{d}}{2N_{d}}\right) |\FT{f}(\kappa)|^{2} \\
        \leq \sum_{\varepsilon \in \{0,1\}^{d}} \sum_{\kappa_{1} \in I_{\varepsilon_{1}}(N_{1})} \dots \sum_{\kappa_{d} \in I_{\varepsilon_{d}}(N_{d})} \left(\frac{\pi \kappa_{1}}{2N_{1}}\right)^{2\varepsilon_{1}} \dots \left(\frac{\pi \kappa_{d}}{2N_{d}}\right)^{2\varepsilon_{d}} |\FT{f}(\kappa)|^{2} \\
        = \sum_{\varepsilon \in \{0,1\}^{d}} O\left( N_{1}^{2 \varepsilon_{1} - 2 \alpha_{1} - 2 \varepsilon_{1}} \dots N_{d}^{2 \varepsilon_{d} - 2 \alpha_{d} - 2 \varepsilon_{d}} \right)
        = O\left(N_{1}^{-2 \alpha_{1}} \dots N_{d}^{-2 \alpha_{d}}\right),
    \end{multline*}
    which is the required estimate.
\end{proof}

\section{Boundedness of Fourier multipliers on H\"older-Lipschitz spaces}
\label{s:FourierMult}

Boundedness results for Fourier multipliers on H\"older-Lipschitz spaces derived from the second Titchmarsh theorem with respect to a norm related to the asymptotic estimate in the said theorem will be presented. Firstly, the result for additive H\"older-Lipschitz spaces is considered, and subsequently applied to prove Lipschitz-Sobolev regularity for Bessel potential operators. This is followed by a detailed comparison of the introduced norm with another one that appears in the literature. Secondly, the case of multiplicative H\"older-Lipschitz spaces is concisely discussed.

As mentioned, the second Titchmarsh theorem characterizes additive H\"older-Lipschitz functions as $L^2(\Omega)$-functions satisfying the asymptotic estimate $\sum_{\abs{\kappa} > N} |\FT{f}(\kappa)|^2 = O\left(N^{-2\alpha}\right)$. This estimate hints to a norm on $\HoldLipSpaceFD{\alpha}{2}$.

\begin{definition}
	Let $0 < \alpha < 1$. We define the norm \(\normIn{\cdot}{\HoldLipSpaceFD{\alpha}{2}}\) on the additive H\"older-Lipschitz space \(\HoldLipSpaceFD{\alpha}{2}\) by
	\[\normIn{f}{\HoldLipSpaceFD{\alpha}{2}}
	:= \normIn{f}{L^2(\Omega)} + \sup_{N \in \bN} N^{\alpha} \left(\sum_{\abs{\kappa} > N} |\FT{f}(\kappa)|^{2}\right)^{\frac{1}{2}}.\]
\end{definition}

We now discuss a boundedness result for Fourier multipliers on (additive) H\"older-Lipschitz spaces with respect to the newly introduced norm under a growth condition on the symbol of the Fourier multiplier.

\begin{theorem}\label{th:FM}
	Let \(\gamma \in \bR\) with \(0 \leq \gamma < 1\). Suppose that for some $C > 0$ the function \(\sigma : \dualLat{L} \to \bC\) satisfies the growth estimate
	\[\abs{\sigma(\kappa)} \leq C \langle \kappa \rangle^{-\gamma} \quad \text{with} \quad \langle \kappa \rangle := (1+\abs{\kappa}^{2})^{\frac{1}{2}}.\]
	Let \(T\) be the Fourier multiplier with symbol \(\sigma\), i.e., \(\FT{Tf}(\kappa) = \sigma(\kappa) \FT{f}(\kappa)\) for all \(\kappa \in \dualLat{L}\). Then
	\[T : \HoldLipSpaceFD{\alpha}{2} \to \HoldLipSpaceFD{\alpha + \gamma}{2}\]
	is bounded for every \(\alpha \in \bR\) with \(0 < \alpha < 1 - \gamma\).
\end{theorem}

\begin{proof}
	Let \(f \in \HoldLipSpace{\Omega}{\alpha}{2}\). By \cref{th:secondTitchmarsh} we have
	\begin{align}
		\sum_{\abs{\kappa} > N} \abs{\FT{Tf}(\kappa)}^{2} = \sum_{\abs{\kappa} > N} \abs{\sigma(\kappa)}^{2} |\FT{f}(\kappa)|^{2}
		&\leq C^2 \sum_{\abs{\kappa} > N} \langle \kappa \rangle^{-2\gamma} |\FT{f}(\kappa)|^{2} \nonumber \\
		&\leq C^2 N^{-2\gamma} \sum_{\abs{\kappa} > N} |\FT{f}(\kappa)|^{2} \label{eq:HoldLipFM1} \\
		&= O\left(N^{-2(\alpha + \gamma)}\right), \label{eq:HoldLipFM2}
	\end{align}
	where we used that $\abs{\kappa} > N$ implied that
	\[\langle \kappa \rangle^{-2\gamma} = \left(\frac{1}{1+\abs{\kappa}^{2}}\right)^{\gamma} \leq \left(\frac{1}{1 + N^2}\right)^{\gamma} \leq N^{-2\gamma}.\]
	By \cref{th:secondTitchmarsh}, we find from \eqref{eq:HoldLipFM2} that \(T(\HoldLipSpace{\Omega}{\alpha}{2}) \subseteq \HoldLipSpace{\Omega}{\alpha+\gamma}{2}\) for all \(\alpha > 0\) with \(\alpha + \gamma < 1\). Furthermore, \(T : L^{2}(\Omega) \to L^{2}(\Omega)\) is bounded because the symbol \(\sigma\) is bounded, since Plancherel's theorem implies that
	\begin{equation}\label{eq:HoldLipFM3}
		\normIn{Tf}{L^2(\Omega)}
		= \lVert \FT{Tf} \rVert_{\ell^2(\dualLat{L})}
		= \lVert \sigma \FT{f} \rVert_{\ell^2(\dualLat{L})}
		\leq \normIn{\sigma}{\ell^{\infty}(\dualLat{L})} \lVert \FT{f} \rVert_{\ell^2(\dualLat{L})}
		= \normIn{\sigma}{\ell^{\infty}(\dualLat{L})} \normIn{f}{L^2(\Omega)}.
	\end{equation}
	Thus, we find using \eqref{eq:HoldLipFM1} and \eqref{eq:HoldLipFM3} that
	\begin{align*}
		\normIn{Tf}{\HoldLipSpaceFD{\alpha + \gamma}{2}}
		&= \normIn{Tf}{L^2(\Omega)} + \sup_{N \in \bN} N^{\alpha + \gamma} \left( \sum_{\abs{\kappa} > N} \abs{\FT{Tf}(\kappa)}^2 \right)^{\frac{1}{2}} \\
		&\leq \normIn{\sigma}{\ell^{\infty}(\dualLat{L})} \normIn{f}{L^2(\Omega)} + C \sup_{N \in \bN} N^{\alpha} \left( \sum_{\abs{\kappa} > N} |\FT{f}(\kappa)|^2 \right)^{\frac{1}{2}} \\
		&\leq \max\left\{ \normIn{\sigma}{\ell^{\infty}(\dualLat{L})}, C \right\} \normIn{f}{\HoldLipSpaceFD{\alpha}{2}},
	\end{align*}
	which shows that \(T : \HoldLipSpace{\Omega}{\alpha}{2} \to \HoldLipSpace{\Omega}{\alpha+\gamma}{2}\) is bounded.
\end{proof}

As an application of \cref{th:FM} we deduce Lipschitz-Sobolev regularity for Bessel potential operators on fundamental domains of lattices. As in \cite{LpLqBoundedness}, we denote by $\Delta_{\Omega}$ the Fourier multiplier with symbol $-4 \pi^2 \abs{\kappa}^2$.

\begin{corollary}
	For every \(0 \leq \gamma < 1\) and \(0 < \alpha < 1 - \gamma\) we have the continuous embedding
	\[H^{\gamma} \HoldLipSpaceFD{\alpha}{2} \hookrightarrow \HoldLipSpaceFD{\alpha+\gamma}{2},\]
	where
	\[H^{\gamma} \HoldLipSpaceFD{\alpha}{2} := \left\{f \in \HoldLipSpaceFD{\alpha}{2} : (I-\Delta_{\Omega})^{\gamma/2} f \in \HoldLipSpaceFD{\alpha}{2}\right\}\]
	with norm
	\[\normIn{f}{H^{\gamma} \HoldLipSpaceFD{\alpha}{2}} := \normIn{(I-\Delta_{\Omega})^{\frac{\gamma}{2}} f}{\HoldLipSpaceFD{\alpha}{2}}.\]
\end{corollary}

\begin{proof}
	Let $0 \leq \gamma < 1$ and $0 < \alpha < 1 - \gamma$ be arbitrary. Applying \cref{th:FM} to the Fourier multiplier \((I-\Delta_{\Omega})^{-\gamma/2}\), we obtain that there exists a constant $C > 0$ such that for $g \in \HoldLipSpaceFD{\alpha}{2}$,
	\[\normIn{(I-\Delta_{\Omega})^{-\frac{\gamma}{2}}g}{\HoldLipSpaceFD{\alpha+\gamma}{2}} \leq C \normIn{g}{\HoldLipSpaceFD{\alpha}{2}}.\]
	Substituting \((I-\Delta_{\Omega})^{\frac{\gamma}{2}} f\) with $f \in \HoldLipSpaceFD{\alpha + \gamma}{2}$ for \(g\), we obtain
	\[\normIn{f}{\HoldLipSpaceFD{\alpha+\gamma}{2}} \leq C \normIn{(I-\Delta_{\Omega})^{\frac{\gamma}{2}}f}{\HoldLipSpaceFD{\alpha}{2}},\]
	which yields the result.
\end{proof}

Let us now turn to a discussion of the norm $\normIn{\cdot}{\HoldLipSpaceFD{\alpha}{2}}$. In some literature, such as in \cite{LipThesis} and \cite{LipArticle}, the (additive) H\"older-Lipschitz space $\HoldLipSpaceFD{\alpha}{p}$ with $0 < \alpha \leq 1$ and $1 < p \leq \infty$ is endowed with the norm
\[\normIn{f}{\HoldLipSpaceFD{\alpha}{p}}' := \normIn{f}{L^{p}(\Omega)} + \sup_{\abs{h} \neq 0} \abs{h}^{-\alpha} \normIn{f(\cdot+h)-f(\cdot)}{L^{p}(\Omega)},\]
which turns $\HoldLipSpaceFD{\alpha}{p}$ into a Banach space. Our goal is to show that in the case $0 < \alpha < 1$ and $p=2$ this norm is actually equivalent with our norm $\normIn{\cdot}{\HoldLipSpaceFD{\alpha}{2}}$. To this end we need a stronger formulation of the one-dimensional case of Duren's lemma \ref{lemma:Duren}.

To set the stage for the refined formulation of Duren's lemma, let us first recall the summation by parts formula.

\begin{lemma}[Summation by parts]\label{lemma:SBP}
	Suppose \(f : \bN \to \bC\) and \(g : \bN \to \bC\). Define the forward difference operator $\diffOpForw$ by $\diffOpForw f(n) := f(n+1) - f(n)$, and define the backward difference operator $\diffOpBack$ by $\diffOpBack f(n) := f(n) - f(n-1)$.
	Then, for all \(N, M \in \bN\) with $1 \leq N < M$,
	\[
	\sum_{k=N}^{M} f(k) \, \diffOpBack g(k) = f(M) g(M) - f(N) g(N-1) - \sum_{k=N}^{M-1} \diffOpForw f(k) \, g(k).
	\]
\end{lemma}

\begin{proof}
	The proof of summation by parts can be found in \cite[Equation (3a)]{SBP}, for example. However, since the proof is simple, we will show it here for the convenience of the reader.
	
	A straightforward calculation gives
	\begin{align*}
		\sum_{k=N}^{M} f(k) \, \diffOpBack g(k)
		&= \sum_{k=N}^{M} f(k) g(k) - \sum_{k=N}^{M} f(k) g(k-1) \\
		&= \sum_{k=N}^{M} f(k) g(k) - \sum_{k=N-1}^{M-1} f(k+1) g(k) \\
		&= f(M) g(M) - f(N) g(N-1) - \sum_{k=N}^{M-1} \diffOpForw f(k) \, g(k),
	\end{align*}
	which shows the lemma.
\end{proof}

Let us now prove a refined version of Duren's lemma, which gives explicit constants for the asymptotic estimates.

\begin{lemma}[Duren, refined formulation]\label{lemma:DurenRefined}
	Consider a non-negative sequence $c_{n} \geq 0$, and let $0 < a < b$. If there exist constants $N_0 \in \bN$ and $C > 0$ such that for all $N \geq N_0$,
	\begin{equation}\label{ineq:DurenRefined1}
		\sum_{n \leq N} n^{b} c_{n} \leq C N^a,
	\end{equation}
	then there is a constant $K_{a,b} > 0$ such that for all $N \geq N_0$,
	\[\sum_{n > N} c_n \leq K_{a,b} C N^{a-b}.\]
	Conversely, if there exist constants $N'_0 \in \bN$ and $C' > 0$ such that for all $N \geq N'_0$,
	\begin{equation}\label{ineq:DurenRefined2}
		\sum_{n > N} c_n \leq C' N^{a-b},
	\end{equation}
	then there exists a constant $K'_{a,b} > 0$ such that for all $N \in \bN \setminus \{0\}$,
	\[\sum_{n \leq N} n^b c_n \leq \left( \normIn{c}{\ell^{1}(\bN \setminus \{0\})} + K'_{a,b} C' \right) N^{a}.\]
\end{lemma}

\begin{proof}
	Suppose that \eqref{ineq:DurenRefined1} holds. Set $S_n := \sum_{k=1}^{n} k^{b} c_k$. Note that $\diffOpBack S_n = n^b c_n$ so that an application of \cref{lemma:SBP} gives for $N_0 \leq N < M - 1$ that
	\begin{align}\label{ineq:DurenRefined3}
		\sum_{n=N+1}^{M} c_n &= \sum_{n=N+1}^{M} n^{-b} \diffOpBack S_n \nonumber \\
		&= M^{-b} S_{M} - (N+1)^{-b} S_{N} - \sum_{n=N+1}^{M-1} \left( (n+1)^{-b} - n^{-b} \right) S_n \nonumber \\
		&\leq C M^{a-b} + b C \sum_{n=N+1}^{M-1} n^{a-b-1},
	\end{align}
	where we used that
	\begin{equation}\label{ineq:DurenRefined6}
		n^{-b} - (n+1)^{-b} = b \int_{n}^{n+1} x^{-b-1} \, \d{x} \leq b n^{-b-1}.
	\end{equation}
	Taking the limit $M \to \infty$ in \eqref{ineq:DurenRefined3}, we obtain that
	\[
	\sum_{n=N+1}^{\infty} c_n \leq b C \sum_{n = N+1}^{\infty} n^{a-b-1}
	\leq b C \int_{N}^{\infty} x^{a-b-1} \, \d{x}
	= \frac{b}{b-a} C N^{a-b}.
	\]
	
	Conversely, assume that \eqref{ineq:DurenRefined2} holds. Set $R_n := \sum_{k > n} c_k$ and note that $\diffOpBack R_n = -c_n$. Remark that since $R_N \leq C' N^{a-b}$ for all $N \geq N'_0$, there does exist a constant $C'' > 0$ such that $R_N \leq C'' N^{a-b}$ for all $N \in \bN \setminus \{0\}$. Hence, without loss of generality, we may assume that \eqref{ineq:DurenRefined2} holds for all $N \in \bN \setminus \{0\}$. Then, for $N \in \bN \setminus \{0\}$ we have
	\begin{equation}\label{ineq:DurenRefined4}
		\sum_{n=1}^{N} n^b c_n
		= \sum_{n=1}^{N} n^b \diffOpBack(-R_n) = -N^b R_N + R_0 + \sum_{n=1}^{N-1} \left( (n+1)^b - n^b \right) R_n.
	\end{equation}
	Now, note that
	\[(n+1)^b - n^b = b \int_{n}^{n+1} x^{b-1} \, \d{x} \leq \begin{dcases}
		b (n+1)^{b-1} \leq b 2^{b-1} n^{b-1} &\quad \text{if} \quad b - 1 \geq 0 \\
		b n^{b-1} & \quad \text{if} \quad b - 1 < 0
	\end{dcases},\]
	where we used that $(n+1)^{b-1} \leq \sup_{m \in \bN \setminus \{0\}} \left(\frac{m+1}{m}\right)^{b-1} n^{b-1} = 2^{b-1} n^{b-1}$ if $b - 1 \geq 0$. Hence, in more compact form, we found for $n \in \bN \setminus \{0\}$ that
	\begin{equation}\label{ineq:DurenRefined7}
		(n+1)^b - n^b \leq b \max\{1, 2^{b-1}\} n^{b-1}
	\end{equation}
	so that \eqref{ineq:DurenRefined4} can be estimated further as
	\begin{equation}\label{ineq:DurenRefined5}
		\sum_{n=1}^{N} n^b c_n \leq R_0 + b \max\{1, 2^{b-1}\} C' \sum_{n=1}^{N-1} n^{a-1}.
	\end{equation}
	We calculate that
	\begin{equation}\label{ineq:DurenRefined8}
		\sum_{n=1}^{N-1} n^{a-1} \leq \begin{dcases}
			\int_{1}^{N} x^{a-1} \, \d{x} = \frac{N^a}{a} - \frac{1}{a} \leq \frac{N^a}{a} &\quad \text{if} \quad a - 1 \geq 0 \\
			\int_{0}^{N-1} x^{a-1} \, \d{x} = \frac{(N-1)^a}{a} \leq \frac{N^a}{a} &\quad \text{if} \quad a - 1 < 0
		\end{dcases},
	\end{equation}
	so that \eqref{ineq:DurenRefined5} becomes
	\[\sum_{n=1}^{N} n^b c_n \leq R_0 + K'_{a,b} C' N^a \leq \left( R_0 + K'_{a,b} C' \right) N^a\]
	for $N \in \bN \setminus \{0\}$, where
	\[K'_{a,b} := \frac{b}{a} \max\{1, 2^{b-1}\},\]
	which concludes the proof.
\end{proof}

The former refined formulation of Duren's lemma can be adapted to fit our setting of lattices.

\begin{lemma}\label{lemma:DurenLatAddRefined}
	Let $d \in \bN \setminus \{0\}$, and let $L$ be a lattice in $\bR^d$. Consider a non-negative sequence $c_{\lambda} \geq 0$ with $\lambda \in L$, and let $0 < a < b$. If there exist constants $N_0 \in \bN$ and $C > 0$ such that for all $N \geq N_{0}$,
	\begin{equation}\label{ineq:DurenLatAddRefined1}
		\sum_{\abs{\lambda} \leq N} \abs{\lambda}^b c_{\lambda} \leq C N^a,
	\end{equation}
	then there is a constant $K_{a,b} > 0$ such that for all $N \geq N_0$,
	\[\sum_{\abs{\lambda} > N} c_{\lambda} \leq K_{a,b} C N^{a-b}.\]
	Conversely, if there exist constant $N'_0 \in \bN$ and $C' > 0$ such that for all $N \geq N'_0$,
	\begin{equation}\label{ineq:DurenLatAddRefined2}
		\sum_{\abs{\lambda} > N} c_{\lambda} \leq C' N^{a-b},
	\end{equation}
	then there exists a constant $K'_{a,b} > 0$ such that for all $N \in \bN \setminus \{0\}$,
	\[\sum_{\abs{\lambda} \leq N} \abs{\lambda}^b c_{\lambda} \leq \left( \normIn{c}{\ell^{1}(\dualLat{L})} + K'_{a,b} C' \right) N^a.\]
\end{lemma}

\begin{proof}
	Firstly, suppose that \eqref{ineq:DurenLatAddRefined1} holds. Let $N \in \bN \setminus \{0\}$. In this case we have that
	\[\sum_{n=1}^N n^b \left( \sum_{n < \abs{\lambda} \leq n+1} c_{\lambda} \right)
	< \sum_{n=0}^N \sum_{n < \abs{\lambda} \leq n+1} \abs{\lambda}^b c_{\lambda}
	\leq \sum_{\abs{\lambda} \leq N+1} \abs{\lambda}^b c_{\lambda}
	\leq C (N+1)^a \leq 2^a C N^a,\]
	where we used that for $N \in \bN \setminus \{0\}$ we have
	\[(N+1)^a \leq \sup_{M \in \bN \setminus \{0\}} \left(\frac{M+1}{M}\right)^a N^a = 2^a N^a.\]
	It follows from \cref{lemma:DurenRefined} that for some constant $\tilde{K}_{a,b} > 0$ we have that
	\[\sum_{\abs{\lambda} > N} c_{\lambda}
	= \sum_{n \geq N} \left( \sum_{n < \abs{\lambda} \leq n+1} c_{\lambda} \right)
	\leq 2^a \tilde{K}_{a,b} C N^{a-b} = K_{a,b} C N^{a-b},\]
	where $K_{a,b} = 2^a \tilde{K}_{a,b}$.
	
	Conversely, if \eqref{ineq:DurenLatAddRefined2} holds, we have that
	\[\sum_{n > N} \left( \sum_{n-1 < \abs{\lambda} \leq n} c_{\lambda} \right)
	= \sum_{\abs{\lambda} > N} c_{\lambda}
	\leq C' N^{a-b}\]
	so that \cref{lemma:DurenRefined} asserts the existence of a constant $K'_{a,b} > 0$ such that for all $N \in \bN \setminus \{0\}$,
	\[\sum_{\abs{\lambda} \leq N} \abs{\lambda}^{b} c_{\lambda}
	\leq \sum_{n=1}^{N} n^{b} \left( \sum_{n-1 < \abs{\lambda} \leq n} c_{\lambda} \right)
	\leq \left( \normIn{c}{\ell^{1}(\dualLat{L})} + K'_{a,b} C' \right) N^a,\]
	proving the lemma.
\end{proof}

The former lattice version of the refined formulation of Duren's lemma allows us to prove that the norms $\normIn{\cdot}{\HoldLipSpaceFD{\alpha}{2}}$ and $\normIn{\cdot}{\HoldLipSpaceFD{\alpha}{2}}'$ on $\HoldLipSpaceFD{\alpha}{2}$ for $0 < \alpha < 1$ are equivalent, since the refinement is essentially an explicit form for the constants involved in the $O$-estimates, so that we can now establish uniform estimates.

\begin{lemma}
	The norms
	\[\normIn{f}{\HoldLipSpaceFD{\alpha}{2}}
	= \normIn{f}{L^2(\Omega)} + \sup_{N \in \bN} N^{\alpha} \left(\sum_{\abs{\kappa} > N} |\FT{f}(\kappa)|^{2}\right)^{\frac{1}{2}}\]
	and
	\[\normIn{f}{\HoldLipSpaceFD{\alpha}{2}}'
	:= \normIn{f}{L^2(\Omega)} + \sup_{\abs{h} \neq 0} \frac{\normIn{f(\cdot+h) - f(\cdot)}{L^2(\Omega)}}{\abs{h}^{\alpha}}\]
	on the additive H\"older-Lipschitz space $\HoldLipSpaceFD{\alpha}{2}$ are equivalent.
\end{lemma}

\begin{proof}
	Let $f \in \HoldLipSpaceFD{\alpha}{2}$, and set
	\[A_{f,\alpha} := \sup_{\abs{h} \neq 0} \frac{\normIn{f(\cdot + h) - f(\cdot)}{L^2(\Omega)}}{\abs{h}^{\alpha}}.\]
	Using Jordan's inequality \eqref{ineq:JordanIneq} and Plancherel's identity \eqref{eq:Plancherel} it holds for $h \in \bR^d$ that
	\begin{align*}
		\sum_{\abs{\kappa} \leq \frac{1}{2\abs{h}}} 16 \abs{\kappa \cdot h}^2 |\FT{f}(\kappa)|^2 \leq \sum_{\abs{\kappa} \leq \frac{1}{2\abs{h}}} 4 \abs{\sin(\pi \kappa \cdot h)}^2 |\FT{f}(\kappa)|^2 &= \normIn{\FT{f(\cdot + h) - f(\cdot)}}{L^2(\Omega)}^2 \\
		&= \normIn{f(\cdot + h) - f(\cdot)}{L^2(\Omega)}^2 \\
		&\leq A_{f,\alpha}^2 \abs{h}^{2\alpha}.
	\end{align*}
	Let $N \in \bN \setminus \{0\}$, and set $h_i := \frac{1}{2N} e_i$ for every $1 \leq i \leq d$. Then we get for every $1 \leq i \leq d$ that
	\[\sum_{\abs{\kappa} \leq N} 4 \frac{\abs{\kappa_i}^2}{N^2} |\FT{f}(\kappa)|^2 \leq A_{f,\alpha}^2 \abs{h_i}^{2\alpha} = A_{f,\alpha}^2 2^{-2\alpha} N^{-2\alpha}\]
	so that summing over $i$ gives that
	\[4 \sum_{\abs{\kappa} \leq N} \frac{\abs{\kappa}^2}{N^2} |\FT{f}(\kappa)|^2 \leq 2^{-2\alpha} d A_{f,\alpha}^2 N^{-2\alpha},\]
	which can be rewritten as
	\[\sum_{\abs{\kappa} \leq N} \abs{\kappa}^2 |\FT{f}(\kappa)|^2 \leq 2^{-2\alpha-2} d A_{f,\alpha}^2 N^{2(1-\alpha)}.\]
	Then \cref{lemma:DurenLatAddRefined} implies that for some $K_{\alpha} > 0$ we have
	\[\sum_{\abs{\kappa} > N} |\FT{f}(\kappa)|^2 \leq K_{\alpha} 2^{-4\alpha} d A_{f,\alpha}^2 N^{-2\alpha}\]
	so that
	\[\sup_{N \in \bN} N^{\alpha} \left( \sum_{\abs{\kappa} > N} |\FT{f}(\kappa)|^2 \right)^{\frac{1}{2}} \leq 2^{-2\alpha} \sqrt{K_{\alpha} d} A_{f,\alpha},\]
	proving that
	\[\normIn{f}{\HoldLipSpaceFD{\alpha}{2}} \leq \max\left\{1, 2^{-2\alpha} \sqrt{K_{\alpha} d}\right\} \normIn{f}{\HoldLipSpaceFD{\alpha}{2}}'.\]
	
	Conversely, set
	\[B_{f,\alpha} := \sup_{N \in \bN} N^{\alpha} \left( \sum_{\abs{\kappa} > N} |\FT{f}(\kappa)|^2 \right)^{\frac{1}{2}}.\]
	Let $h \in \bR^d \setminus \{0\}$ be arbitrary, and set $N_h := \floor{\frac{1}{\abs{h}}}$. Note that $\frac{1}{2\abs{h}} \leq N_h \leq \frac{1}{\abs{h}}$ if $N_h \geq 1$. It follows from Plancherel's identity \eqref{eq:Plancherel} that
	\begin{align}\label{ineq:equivNorms1}
		\normIn{f(\cdot + h) - f(\cdot)}{L^2(\Omega)}^{2}
		&= \sum_{\kappa \in \dualLat{L}} 4 \sin^2(\pi \kappa \cdot h) |\FT{f}(\kappa)|^2 \nonumber \\
		&\leq 4 \pi^2 \abs{h}^2 \sum_{\abs{\kappa} \leq N_h} \abs{\kappa}^2 |\FT{f}(\kappa)|^2 + 4 \sum_{\abs{\kappa} > N_h} |\FT{f}(\kappa)|^2.
	\end{align}
	Note that for $N \in \bN \setminus \{0\}$ we have that
	\begin{equation}\label{ineq:equivNorms2}
		\sum_{\abs{\kappa} > N} |\FT{f}(\kappa)|^2 \leq B_{f,\alpha}^2 N^{-2\alpha}
	\end{equation}
	so that \cref{lemma:DurenLatAddRefined} gives that there exists a constant $K'_{\alpha} > 0$ such that
	\begin{equation}\label{ineq:equivNorms3}
		\sum_{\abs{\kappa} \leq N} \abs{\kappa}^2 |\FT{f}(\kappa)|^2 \leq \left( \lVert \FT{f} \rVert_{\ell^{2}(\dualLat{L})}^2 + K'_{\alpha} B_{f,\alpha}^2 \right) N^{2(1-\alpha)}
		= \left( \lVert f \rVert_{L^{2}(\Omega)}^2 + K'_{\alpha} B_{f,\alpha}^2 \right) N^{2(1-\alpha)}.
	\end{equation}
	Hence, \eqref{ineq:equivNorms1} can be further estimated with the help of \eqref{ineq:equivNorms2} and \eqref{ineq:equivNorms3} as
	\begin{align*}
		\normIn{f(\cdot + h) - f(\cdot)}{L^2(\Omega)}^2
		&\leq 4 \pi^2 \abs{h}^2 \left( \lVert f \rVert_{L^{2}(\Omega)}^2 + K'_{\alpha} B_{f,\alpha}^2 \right) N_h^{2(1-\alpha)} + 4 B_{f,\alpha}^2 N_h^{-2\alpha} \\
		&\leq \left( 4 \pi^2 \normIn{f}{L^2(\Omega)}^2 + \left( 4 \pi^2 K'_{\alpha} + 2^{2 + 2 \alpha} \right) B_{f,\alpha}^2 \right) \abs{h}^{2\alpha}
	\end{align*}
	so that
	\begin{align*}
		\sup_{\abs{h} \neq 0} \frac{\normIn{f(\cdot + h) - f(\cdot)}{L^2(\Omega)}}{\abs{h}^{\alpha}}
		&\leq \sqrt{4 \pi^2 \normIn{f}{L^2(\Omega)}^2 + \left( 4 \pi^2 K'_{\alpha} + 2^{2 + 2 \alpha} \right) B_{f,\alpha}^2} \\
		&\leq 2 \pi \normIn{f}{L^2(\Omega)} + \sqrt{4 \pi^2 K'_{\alpha} + 2^{2 + 2\alpha}} B_{f,\alpha}.
	\end{align*}
	Thus, we find that
	\[\normIn{f}{\HoldLipSpaceFD{\alpha}{2}}' \leq \max\left\{1 + 2\pi, \sqrt{4 \pi^2 K'_{\alpha} + 2^{2+2\alpha}}\right\} \normIn{f}{\HoldLipSpaceFD{\alpha}{2}},\]
	concluding the proof.
\end{proof}

Next, we concisely deal with boundedness of Fourier multipliers on multiplicative H\"older-Lipschitz spaces. Let us start by introducing a norm on $\HoldLipSpaceMultFD{\alpha_1, \dots, \alpha_d}{2}$ corresponding to the asymptotic estimate in \cref{th:secondTitchMult}.

\begin{definition}
	Let \(0 < \alpha_{1}, \dots, \alpha_{d} < 1\). We define the norm \(\normIn{\cdot}{\HoldLipSpaceMultFD{\alpha_1, \dots, \alpha_d}{2}}\) on the multiplicative H\"older-Lipschitz space \(\HoldLipSpaceMultFD{\alpha_1, \dots, \alpha_d}{2}\) by
	\[\normIn{f}{\HoldLipSpaceMultFD{\alpha_1, \dots, \alpha_d}{2}}
	:= \normIn{f}{L^2(\Omega)} + \sup_{N_1, \dots, N_d \in \bN} N_1^{\alpha_1} \dots N_{d}^{\alpha_d} \left(\sum_{\abs{\kappa_{1}} > N_1} \dots \sum_{\abs{\kappa_{d}} > N_d} |\FT{f}(\kappa)|^{2}\right)^{\frac{1}{2}}.\]
\end{definition}

We can derive a boundedness result for Fourier multipliers on multiplicative H\"older-Lipschitz spaces in a similar manner as for the additive case.

\begin{theorem}
	Let \(\gamma_{1}, \dots, \gamma_{d} \in \bR\) with \(0 \leq \gamma_{1}, \dots, \gamma_{d} < 1\). Suppose that for some $C > 0$ the function \(\sigma : \dualLat{L} \to \bC\) satisfies the growth estimate
	\[\abs{\sigma(\kappa)} \leq C \langle \kappa_{1} \rangle^{-\gamma_{1}} \dots \langle \kappa_{d} \rangle^{-\gamma_{d}} \quad \text{with} \quad \langle \kappa_{j} \rangle := (1+\abs{\kappa_{j}}^{2})^{\frac{1}{2}}.\]
	Let \(T\) be the Fourier multiplier with symbol \(\sigma\), i.e.\ \(\FT{Tf}(\kappa) = \sigma(\kappa) \FT{f}(\kappa)\) for all \(\kappa \in \dualLat{L}\). Then
	\[T : \HoldLipSpaceMultFD{\alpha_{1}, \dots, \alpha_{d}}{2} \to \HoldLipSpaceMultFD{\alpha_{1} + \gamma_{1}, \dots, \alpha_{d} + \gamma_{d}}{2}\]
	is bounded for every \(\alpha_{1}, \dots, \alpha_{d} \in \bR\) with \(0 < \alpha_{j} < 1 - \gamma_{j}\) for all \(1 \leq j \leq d\).
\end{theorem}

\begin{proof}
	     Let \(f \in \HoldLipSpaceMultFD{\alpha_1, \dots, \alpha_d}{2}\). By \cref{th:secondTitchMult} we have
	     \begin{align}
	     	\sum_{\abs{\kappa_1} > N_1} \dots \sum_{\abs{\kappa_d} > N_d} \abs{\FT{Tf}(\kappa)}^{2}
	     	&= \sum_{\abs{\kappa_1} > N_1} \dots \sum_{\abs{\kappa_d} > N_d} \abs{\sigma(\kappa)}^{2} |\FT{f}(\kappa)|^{2} \nonumber \\
	     	&\leq C^2 \sum_{\abs{\kappa_1} > N_1} \dots \sum_{\abs{\kappa_d} > N_d} \langle \kappa_1 \rangle^{-2\gamma_1} \dots \langle \kappa_d \rangle^{-2\gamma_d} |\FT{f}(\kappa)|^{2} \nonumber \\
	     	&\leq C^2 N_{1}^{-2\gamma_{1}} \dots N_{d}^{-2\gamma_{d}} \sum_{\abs{\kappa_1} > N_1} \dots \sum_{\abs{\kappa_d} > N_d} |\FT{f}(\kappa)|^{2} \label{eq:HoldLipMultFM1} \\
	     	&= O\left(N_{1}^{-2(\alpha_1 + \gamma_1)} \dots N_{d}^{-2(\alpha_d + \gamma_d)}\right). \nonumber
	     \end{align}
	     Consequently, \(T\left(\HoldLipSpaceMultFD{\alpha_1, \dots, \alpha_d}{2}\right) \subseteq \HoldLipSpaceMultFD{\alpha_1+\gamma_1, \dots, \alpha_d+\gamma_d}{2}\) for all \(\alpha_1, \dots, \alpha_d > 0\) with \(\alpha_j + \gamma_j < 1\) for every $1 \leq j \leq d$ because of \cref{th:secondTitchMult}. Next, \(T : L^{2}(\Omega) \to L^{2}(\Omega)\) is bounded because the symbol \(\sigma\) is bounded, as shown in \eqref{eq:HoldLipFM3}.
	     Thus, we find using \eqref{eq:HoldLipMultFM1} that
	     \begin{align*}
	     	&\normIn{Tf}{\HoldLipSpaceMultFD{\alpha_1 + \gamma_1, \dots, \alpha_d + \gamma_d}{2}} \\
	     	&\qquad= \normIn{Tf}{L^2(\Omega)} + \sup_{N_1, \dots, N_d \in \bN} N_{1}^{\alpha_1 + \gamma_1} \dots N_{d}^{\alpha_d + \gamma_d} \left( \sum_{\abs{\kappa_{1}} > N_1} \dots \sum_{\abs{\kappa_{d}} > N_d} \abs{\FT{Tf}(\kappa)}^2 \right)^{\frac{1}{2}} \\
	     	&\qquad\leq \normIn{\sigma}{\ell^{\infty}(\dualLat{L})} \normIn{f}{L^2(\Omega)} + C \sup_{N_1, \dots, N_d \in \bN} N_{1}^{\alpha_1} \dots N_{d}^{\alpha_d} \left( \sum_{\abs{\kappa_{1}} > N_1} \dots \sum_{\abs{\kappa_{d}} > N_d} |\FT{f}(\kappa)|^2 \right)^{\frac{1}{2}} \\
	     	&\qquad\leq \max\left\{ \normIn{\sigma}{\ell^{\infty}(\dualLat{L})}, C \right\} \normIn{f}{\HoldLipSpaceMultFD{\alpha_1, \dots, \alpha_d}{2}},
	     \end{align*}
	     which shows that \(T : \HoldLipSpaceMultFD{\alpha_1, \dots, \alpha_d}{2} \to \HoldLipSpace{\Omega}{\alpha_1+\gamma_1, \dots, \alpha_d+\gamma_d}{2}\) is bounded.
\end{proof}

It is a natural question to ask whether the refined formulation for Duren's lemma \ref{lemma:DurenRefined} can be extended to several variables so that it fits our multiplicative framework. It is straightforward to show that one implication can directly be generalized to higher dimensions via mathematical induction, while the other one needs very tedious calculations because of the inclusion of $\normIn{c}{\ell^1(\bN \setminus \{0\})}$ in the constant. This problem can be a topic for further investigation.

\section*{Acknowledgments}

I would like to express my thanks to Prof.\ Michael Ruzhansky for suggesting this research topic and his guidance, and to Dr. Vishvesh Kumar for helpful discussions at the start of this project.

\printbibliography

@article{LpLqBoundedness,
      title={$L^{p}$-$L^{q}$ boundedness of {F}ourier multipliers on fundamental domains of lattices in $\mathbb{R}^d$}, 
      author={A. Hendrickx},
      journal = {J. Fourier Anal. Appl.},
      volume = {28},
      number = {60},
      year={2022},
      DOI = {https://doi.org/10.1007/s00041-022-09955-1},
}

@book {LipThesis,
    AUTHOR = {Younis, M. S.},
     TITLE = {{F}ourier Transforms of {L}ipschitz Functions on Compact Groups},
      NOTE = {Ph.D. Thesis -- McMaster University (Canada)},
 PUBLISHER = {ProQuest LLC, Ann Arbor, MI},
      YEAR = {1974},
       URL = {https://www.proquest.com/docview/302766437},
}

@article {LipArticle,
    AUTHOR = {Daher, R. and Delgado, J. and Ruzhansky, M.},
     TITLE = {{T}itchmarsh theorems for {F}ourier transforms of
              {H}\"{o}lder-{L}ipschitz functions on compact homogeneous
              manifolds},
   JOURNAL = {Monatsh. Math.},
  FJOURNAL = {Monatshefte f\"{u}r Mathematik},
    VOLUME = {189},
      YEAR = {2019},
    NUMBER = {1},
     PAGES = {23--49},
      ISSN = {0026-9255},
       DOI = {https://doi.org/10.1007/s00605-018-1253-0},
}

@article {JordanInequality,
    AUTHOR = {Qi, F. and Niu, D.-W. and Guo, B.-N.},
     TITLE = {Refinements, generalizations, and applications of {J}ordan's
              inequality and related problems},
   JOURNAL = {J. Inequal. Appl.},
  FJOURNAL = {Journal of Inequalities and Applications},
      YEAR = {2009},
     PAGES = {Art. ID 271923, 52 pages},
      ISSN = {1025-5834},
       DOI = {https://doi.org/10.1155/2009/271923},
}

@book {Duren,
    AUTHOR = {Duren, P. L.},
     TITLE = {Theory of {$H^{p}$} Spaces},
    SERIES = {Pure and Applied Mathematics},
    VOLUME = {38},
 PUBLISHER = {Academic Press, New York-London},
      YEAR = {1970},
     PAGES = {xii+258},
}

@book {Titchmarsh,
    AUTHOR = {Titchmarsh, E. C.},
     TITLE = {Introduction to the Theory of {F}ourier Integrals},
   EDITION = {3},
 PUBLISHER = {Chelsea Publishing Co., New York},
      YEAR = {1986},
     PAGES = {x+394},
      ISBN = {0-8284-0324-4},
}

@article {LipDom,
    AUTHOR = {Younis, M. S.},
     TITLE = {The {F}ourier transforms of {L}ipschitz functions on certain
              domains},
   JOURNAL = {Internat. J. Math. Math. Sci.},
  FJOURNAL = {International Journal of Mathematics and Mathematical
              Sciences},
    VOLUME = {20},
      YEAR = {1997},
    NUMBER = {4},
     PAGES = {817--822},
      ISSN = {0161-1712},
       DOI = {https://doi.org/10.1155/S0161171297001117},
}

@article {Bray,
    AUTHOR = {Bray, W. O.},
     TITLE = {Growth and integrability of {F}ourier transforms on
              {E}uclidean space},
   JOURNAL = {J. Fourier Anal. Appl.},
  FJOURNAL = {The Journal of Fourier Analysis and Applications},
    VOLUME = {20},
      YEAR = {2014},
    NUMBER = {6},
     PAGES = {1234--1256},
      ISSN = {1069-5869},
       DOI = {https://doi.org/10.1007/s00041-014-9354-1},
}

@article {SBP,
    AUTHOR = {Chu, W.},
     TITLE = {Abel's lemma on summation by parts and basic hypergeometric
              series},
   JOURNAL = {Adv. in Appl. Math.},
  FJOURNAL = {Advances in Applied Mathematics},
    VOLUME = {39},
      YEAR = {2007},
    NUMBER = {4},
     PAGES = {490--514},
      ISSN = {0196-8858},
       DOI = {https://doi.org/10.1016/j.aam.2007.02.001},
}

@book {Zygmund,
    AUTHOR = {Zygmund, A.},
     TITLE = {Trigonometrical Series},
 PUBLISHER = {Dover Publications, New York},
      YEAR = {1955},
     PAGES = {vii+329},
}

@article {TitchmarshSpherical,
    AUTHOR = {El Ouadih, S. and Daher, R.},
     TITLE = {On spherical analogues of the classical theorems of
              {T}itchmarsh},
   JOURNAL = {Integral Transforms Spec. Funct.},
  FJOURNAL = {Integral Transforms and Special Functions. An International
              Journal},
    VOLUME = {31},
      YEAR = {2020},
    NUMBER = {12},
     PAGES = {1010--1019},
      ISSN = {1065-2469},
       DOI = {https://doi.org/10.1080/10652469.2020.1784162},
}

@article {Platonov,
    AUTHOR = {Platonov, S. S.},
     TITLE = {The {F}ourier transform of functions satisfying a {L}ipschitz
              condition on symmetric spaces of rank 1},
   JOURNAL = {Sibirsk. Mat. Zh.},
  FJOURNAL = {Rossi\u{\i}skaya Akademiya Nauk. Sibirskoe Otdelenie. Institut
              Matematiki im. S. L. Soboleva. Sibirski\u{\i} Matematicheski\u{\i}
              Zhurnal},
    VOLUME = {46},
      YEAR = {2005},
    NUMBER = {6},
     PAGES = {1374--1387},
      ISSN = {0037-4474},
       DOI = {https://doi.org/10.1007/s11202-005-0105-z},
}

@article {TitchMarshHyperbolic,
    AUTHOR = {Younis, M. S.},
     TITLE = {Fourier transforms of {L}ipschitz functions on the hyperbolic
              plane {$H^2$}},
   JOURNAL = {Internat. J. Math. Math. Sci.},
  FJOURNAL = {International Journal of Mathematics and Mathematical
              Sciences},
    VOLUME = {21},
      YEAR = {1998},
    NUMBER = {2},
     PAGES = {397--401},
      ISSN = {0161-1712},
       DOI = {https://doi.org/10.1155/S0161171298000544},
}

@article {DamekRicci,
    AUTHOR = {El Ouadih, S. and Daher, R.},
     TITLE = {Lipschitz conditions in {D}amek-{R}icci spaces},
   JOURNAL = {C. R. Math. Acad. Sci. Paris},
  FJOURNAL = {Comptes Rendus Math\'{e}matique. Acad\'{e}mie des Sciences. Paris},
    VOLUME = {359},
      YEAR = {2021},
     PAGES = {675--685},
      ISSN = {1631-073X},
       DOI = {https://doi.org/10.5802/crmath.211},
}

@article{TitchmarshNAgroups,
    title={Titchmarsh theorems, Hausdorff-Young-Paley inequality and $L^p$-$L^q$ boundedness of Fourier multipliers on harmonic $NA$ groups},
    author={Kumar, V. and Ruzhansky, M.},
    year={2021},
    eprint={2107.13044},
    archivePrefix={arXiv},
    primaryClass={math.FA},
    note = {Preprint},
}

@article {LaguerreHypergroup,
    AUTHOR = {Negzaoui, S.},
     TITLE = {Lipschitz conditions in {L}aguerre hypergroup},
   JOURNAL = {Mediterr. J. Math.},
  FJOURNAL = {Mediterranean Journal of Mathematics},
    VOLUME = {14},
      YEAR = {2017},
    NUMBER = {5},
     PAGES = {Paper No. 191, 12},
      ISSN = {1660-5446},
       DOI = {https://doi.org/10.1007/s00009-017-0989-4},
}

@article{FourierWalsh,
  title={An Analog of Titchmarsh’s Theorem for the Fourier–Walsh Transform},
  author={S. S. Platonov},
  journal={Mathematical Notes},
  year={2018},
  volume={103},
  pages={96-103},
  DOI = {https://doi.org/10.1134/S000143461801011X},
}

\end{document}